\documentclass[reqno]{amsart}
 
\usepackage{amsfonts, amsmath,amsthm, amssymb,url,graphicx, color, verbatim,lineno}

\newcommand\NN{\mathbb{N}}
\newcommand\ZZ{\mathbb{Z}}
\newcommand\RR{\mathbb{R}}
\newcommand\FF{\mathbb{F}}

\newcommand\Tor{\operatorname{Tor}}
\newcommand\Ext{\operatorname{Ext}}

\newcommand\Int{\operatorname{Int}}
\newcommand\Res{\operatorname{Res}}
\newcommand\intervals{\operatorname{int}}
\newcommand\link{\operatorname{link}}
\newcommand\red{{\operatorname{{red}}}}

\newcommand\gr{{\mathfrak{gr}}}
\newcommand\mm{{\mathfrak{m}}}

\title[Koszul incidence algebras, semigroups, ideals]{Koszul incidence algebras, affine semigroups, and
Stanley-Reisner ideals}

\author{Victor Reiner}
\address{School of Mathematics\\
University of Minnesota\\
Minneapolis, MN 55455\\
USA}
\email{reiner@math.umn.edu}

\author{Dumitru Ioan Stamate}
\address{Facultatea de Matematic\u{a} \c{s}i Informatic\u{a} \\
Universitatea Bucure\c{s}ti \\
Str. Academiei 14  \\
Bucharest, RO-010014  \\
Romania}
\email{dumitru.stamate@fmi.unibuc.ro}

\theoremstyle{plain}
\newtheorem{theorem}{Theorem}[section]
\newtheorem{corollary}[theorem]{Corollary}
\newtheorem{definition}[theorem]{Definition}
\newtheorem{example}[theorem]{Example}
\newtheorem{proposition}[theorem]{Proposition}
\newtheorem{lemma}[theorem]{Lemma}
\newtheorem{remark}[theorem]{Remark}

\newtheorem{axiom}{Axiom}

\thanks{First author partially supported by NSF grant DMS-0601010.
Second author partially supported by CNCSIS grant ID-PCE no. 51/2007 and a Fulbright Scholarship at the University of Minnesota.}



\begin{document}

\begin{abstract}
We prove a theorem unifying three results from combinatorial homological and commutative
algebra, characterizing the Koszul property for 
incidence algebras of posets and affine semigroup rings, and
characterizing linear resolutions of squarefree monomial ideals. 
The characterization in the graded setting is via
the Cohen-Macaulay property of certain posets or simplicial
complexes, and in the more general nongraded setting, via
the sequential Cohen-Macaulay property.  
\end{abstract}

\maketitle

\tableofcontents

\section{Introduction}
\label{sec:intro}

The main result of this paper, Theorem~\ref{thm:main},
is a characterization  of the Koszul property for certain rings and ideals
within them.
The notion of Koszulness goes back to work of 
Priddy \cite{Pri} and Fr\"oberg \cite{Fr1},
but has been generalized by various authors.  We begin by
reviewing the definition that we will use, incorporating
features introduced by Beilinson, Ginzburg and Soergel \cite{BGS} and 
Fr\"oberg \cite{Fr2}.

\begin{definition}(cf. \cite[Definitions 1.2.1, 2.14.1]{BGS}, \cite[\S 2]{Fr2})
\label{defn:koszulity}
\rm \ \\
Let $R = \bigoplus_{d \geq 0} R_d$ be an $\NN$-graded ring
in which the subring $R_0$ is {\it semisimple}, and
let $M$ be a $\ZZ$-graded (right-) $R$-module.
Say that $M$ is a {\it Koszul $R$-module} if
it admits a graded projective $R$-resolution
$$
\cdots \rightarrow P^{(2)} \rightarrow P^{(1)} \rightarrow P^{(0)} \rightarrow M \rightarrow 0
$$
with each $P^{(i)}$ generated in degree $i$, that is, $P^{(i)}= P^{(i)}_i R$. 
One can define a similar notion for a left $R$-module $M$.
Letting $\mm:=\oplus_{d > 0} R_d$, a two-sided ideal within $R$,
one says that $R$ is a {\it Koszul ring} if the trivially graded
$R$-module $R_0=R/\mm$ is Koszul as a right $R$-module.

More generally, assume $A$ is a ring which is {\it not necessarily
graded} but has a direct sum decomposition
\begin{equation}
\label{choice-of-direct-sum-decomposition}
A=A_0 \oplus I
\end{equation}
in which $A_0$ is a semisimple subring,
and $I$ is a two-sided ideal.  
Let $M$ be a right $A$-module.
Consider the associated graded ring
$
\gr_I A := \oplus_{d \geq 0} I^d A/I^{d+1}A,
$
and the associated graded module
$
\gr_I M := \bigoplus_{d \geq 0} M I^d /M I^{d+1}
$
for the $I$-adic filtrations on $A$ and $M$. 
Say that $M$ is a {\it (nongraded) Koszul} $A$-module
if $\gr_I M$ is a Koszul $\gr_I A$-module in the graded sense 
already defined above.  Note that this notion of nongraded
Koszulness depends implicitly not only on the $A$-module structure of $M$, but also 
on the choice of the decomposition \eqref{choice-of-direct-sum-decomposition} for $A$.
Again, there is a similar notion for a left $A$-module $M$.
Say that $A$ is a {\it (nongraded) Koszul} ring if
the right $A$-module $M=A_0=A/I$ is a {\it (nongraded) Koszul} $A$-module.
\end{definition}

\begin{remark} (on right versus left modules) \rm \ \\
We choose to use
right $R$-modules $M$ rather than left $R$-modules in
order to later reformulate Koszulness (Proposition~\ref{prop:Koszulness-via-tor})
via $\Tor^R(M,k)$ rather
than $\Ext_R(M,k)$.  In particular, Koszulness of the ring $R$ has
been defined here in terms of the {\it right} $R$-module structure on
$R_0=R/\mm$.  However \cite[Prop. 2.2.1]{BGS} shows that $R$ is a Koszul ring if and only
if its opposite $R^{opp}$ is a Koszul ring.  For all the graded rings of interest in this paper
(namely, the associated graded rings $R=\gr_I k[P]_{\red}$ of the reduced 
incidence algebras $k[P]_{\red}$ defined in Definition~\ref{defn:reduced-incidence-algebra}),
one has $R_0$ commutative; see Proposition~\ref{prop:finiteness}.  
Thus for such rings $R$, Koszulness as a ring could be
defined by saying that $R_0$ is Koszul as a right $R$-module, or equivalently, $R_0$ is 
Koszul as a left $R$-module.
\end{remark}

We next review three results 
(Theorems~\ref{thm:Polo-Woodcock},~\ref{thm:PRS}, ~\ref{thm:EHHRW}) 
from the literature on combinatorial homological and commutative algebra which
Theorem~\ref{thm:main} unifies.  Each can be
phrased as characterizing Koszulness for a certain
class of rings or ideals in terms of the {\it Cohen-Macaulay}
property for certain partially ordered sets ({\it posets}) or
simplicial complexes.  For the first two results 
(Theorems~\ref{thm:Polo-Woodcock} and \ref{thm:PRS}), 
our unification will simultaneously remove an unnecessary purity/gradedness hypothesis
by considering {\it sequential Cohen-Macaulayness}, a
natural nonpure generalization of Cohen-Macaulayness
introduced by Stanley \cite[\S III.2.9]{St}.

Throughout the remainder of the paper, $k$ denotes a field.

\subsection{Incidence algebras of posets}

Let $P$ be a poset.
Given two comparable elements $x \leq y$ in $P$, the sets
$$
\begin{aligned}[]
[x,y]&:=\{z \in P: x \leq z \leq y\} \\
(x,y)&:=\{z \in P: x < z < y\} 
\end{aligned}
$$
are called the {\it closed} and {\it open} intervals between $x$ and $y$.

\begin{definition}
\label{defn:incidence-algebra}
\rm \ \\
Assume the poset $P$ is finite, and let
$\Int(P)$ be the set of all closed intervals $[x,y]$.  Then 
the  {\it incidence algebra} $k[P]$ is the 
$k$-vector space having basis $\{ \xi_{[x,y]} \}$ indexed by $\Int(P)$,
with multiplication defined $k$-bilinearly via 
$$
\xi_{[x,y]} \cdot \xi_{[z,w]} = \delta_{y,z} \cdot \xi_{[x,w]},
$$
where $\delta_{y,z}$ is the Kronecker delta.
\end{definition}

For considering Koszulness, decompose
$A=k[P]=A_0 \oplus I$ where $A_0$ is the $k$-span of
$\xi_{[x,x]}$ for $x \in  P$ and $I$ is the
$k$-span of $\xi_{[x,y]}$ with $x<y$ in $P$.

Let $A_1$ denote the $k$-span of $\xi_{[x,y]}$ for
$x<y$ a {\it covering relation} in $P$, that is, $[x,y]=\{x,y\}$.
It is not hard to check that $A=k[P]$ is always generated as an $A_0$-algebra
by $A_1$.  However, there is an $\NN$-grading on $A=k[P]$ having $A_0, A_1$
in degrees $0, 1$ if and only if
$P$ is {\it graded} in the sense that for every interval $[x,y]$,
all of its maximal chains have the same length.

Koszulness for $k[P]$ turns out to be related to a topological property of
its open intervals.  

\begin{definition}
\label{defn:Cohen-Macaulay}
\rm \ \\
Given a poset $P$, its {\it order complex} $\Delta P$ is the abstract simplicial
complex having vertex set $P$ and a face for each totally ordered subset
({\it chain}) in $P$.

Given an abstract simplicial complex $\Delta$, its dimension $\dim \Delta$ is
the maximum dimension of any of its faces.  Say
that $\Delta$ is {\it Cohen-Macaulay over $k$} if for every face $F$ of $\Delta$,
its {\it link}
$$
\link_\Delta(F):=\{ G \in \Delta: F \cup G \in \Delta, F \cap G =\varnothing\}
$$
within $\Delta$ has the property that its reduced homology
vanishes below the dimension of the link:
$$
\tilde{H}_i( \link_\Delta(F) ; k) = 0 \text{ for } -1 \leq i < \dim \link_\Delta(F).
$$
When discussing the Cohen-Macaulayness of a poset $P$, we will always mean
the Cohen-Macaulayness of its order complex $\Delta P$.
\end{definition}

\begin{remark} (on empty simplicial complexes) \rm \ \\
A word is in order about our conventions for the
reduced homology of two special simplicial complexes that
arise as degenerate cases of order complexes for open intervals.
The complex $\{\varnothing\}$ contains the empty face and no other face, 
and is the order complex $\Delta(x,y)$ when $x$ covers $y$, that is, $[x,y]=\{x,y\}$.
It has $\tilde{H}_{-1}(\{ \varnothing\}; k) = k$ and all of its
other reduced homology groups vanish.
Any complex $\Delta$ other than $\{ \varnothing \}$
has $\tilde{H}_{-1}(\Delta;k)=0$.

On the other hand, the empty complex $\varnothing$ contains no faces at all,
and is the order complex $\Delta(x,x)$.
Its homology vanishes entirely:  
$
\tilde{H}_i(\varnothing; k) = 0
$ 
for all $i \geq -1$.
\end{remark}

The first result we wish to unify is the following.

\begin{theorem} (Polo \cite[\S 1.6]{Pol}, Woodcock \cite[Theorem 3.7]{Woc})
\label{thm:Polo-Woodcock}

The incidence algebra $k[P]$ of a finite graded poset $P$
is a Koszul ring if and only if every open interval $(x,y)$ in $P$
is Cohen-Macaulay over $k$.
\end{theorem}

\subsection{Affine semigroup rings}

An {\it affine semigroup} $\Lambda$ is a finitely-generated subsemigroup
of the additive group $(\ZZ^n,+)$.  
The {\it affine semigroup ring} $k\Lambda$ is the 
$k$-vector space having basis $\{ \xi_\lambda \}$ indexed by $\lambda$ in $\Lambda$,
with multiplication defined $k$-bilinearly via 
$\xi_\lambda \cdot \xi_\mu = \xi_{\lambda+\mu}$.
We will assume throughout that $\Lambda$ is {\it pointed},
that is, an element $\lambda \neq 0$ in $\Lambda$ never
has its additive inverse $-\lambda$ in $\Lambda$.  Equivalently,
there exists a linear functional on $\RR^n$ which is strictly positive on
the nonzero elements of $\Lambda$.  

For considering Koszulness, decompose $A=k\Lambda=A_0 \oplus I$ 
where $A_0=k = k\xi_0$ and $I$ is the
$k$-span of $\xi_\lambda$ with $\lambda \neq 0$.

Let $A_1$ denote the $k$-span of $\xi_\lambda$ for those $\lambda$ in $\Lambda$
which are {\it indecomposable}, that is, $\lambda \neq 0$ and  
$\lambda = \lambda_1+\lambda_2$ with
$\lambda_i$ in $\Lambda$ forces either $\lambda_1=0$ or $\lambda_2=0$.
It is not hard to check that $A=k\Lambda$ is always generated as an $A_0$-algebra
by $A_1$.  However, there is an $\NN$-grading on $A=k\Lambda$ having $A_0, A_1$
in degrees $0, 1$ if and only if there exists a 
linear functional on $\RR^n$ whose value on all indecomposable $\lambda$
is $1$. The semigroups having this property are called {\em graded}. 

Koszulness for $k\Lambda$ turns out again to be related to 
poset properties of $\Lambda$, where $\Lambda$ is considered as an infinite poset
in which $\lambda \leq \mu$ if and only if $\mu - \lambda$ lies in $\Lambda$.

The second result we wish to unify is the following.

\begin{theorem} (Peeva, Reiner, and Sturmfels \cite{PRS})
\label{thm:PRS}

A graded pointed affine semigroup ring $k\Lambda$ is a
Koszul ring if and only if every open interval $(\lambda,\mu)$ in the poset $\Lambda$
is Cohen-Macaulay over $k$.
\end{theorem}

\noindent
We remark that in \cite{HRW2}, the authors considered some sufficient 
(but not necessary) conditions for an affine semigroup ring $k\Lambda$ to be 
{\it nongraded} Koszul in the sense considered here.

\subsection{Componentwise linear ideals}

The third result to be unified relates
to resolutions of ideals generated by squarefree monomials within the polynomial algebra
$A=k[x_1,\ldots,x_n]$.  Consider $A$ as an $\NN$-graded ring with its usual
grading in which each variable $x_i$ has degree $1$.

\begin{definition}
\label{defn:linear-resolution}
\rm \ \\
Recall that a $\ZZ$-graded $A$-module $M$ is said to have a {\it linear resolution} as an
$A$-module if there exists an integer $d_0$ and a graded free $A$-resolution
$$
\cdots \rightarrow P^{(2)} \rightarrow P^{(1)} \rightarrow P^{(0)} \rightarrow M \rightarrow 0
$$
such that the free $A$-module $P^{(i)}$ has its basis elements all in degree $d_0+i$.
In particular, this requires that $M$ itself is {\it pure} in the sense that it is generated
in the single degree $d_0$.
\end{definition}

Linearity of resolutions for squarefree monomial ideals
is characterized in terms of the topology of a certain simplicial complex.

\begin{definition}
\label{defn:Stanley-Reisner-ideal}
\rm \ \\
Recall that an ideal in $A$ generated by squarefree monomials is the
{\it Stanley-Reisner ideal} $I_\Delta$ for the unique simplicial
complex $\Delta$ on vertex set $[n]:=\{1,2,\ldots,n\}$ whose simplices
correspond to the squarefree monomials {\it not} lying
in $I_\Delta$.  Recall also that the canonical {\it Alexander dual} $\Delta^\vee$
to the simplicial complex $\Delta$ has the same vertex set
$[n]$, and is defined by
$$
\Delta^\vee:=\{[n] \setminus G: G \not\in \Delta \}.
$$
\end{definition}

\begin{theorem} (Eagon and Reiner \cite[Theorem 3]{ER})
\label{thm:EagonR}

A Stanley-Reisner ideal $I_\Delta$ in $A=k[x_1,\ldots,x_n]$
has linear resolution as $A$-module if and only if 
the Alexander dual complex $\Delta^\vee$ is Cohen-Macaulay over $k$.
\end{theorem}

Herzog and Hibi \cite{HH} introduced an interesting relaxation 
of the notion of linear resolution which is appropriate
for {\it nonpure} $\ZZ$-graded $A$-modules.

\begin{definition}
\rm \ \\
For $A=k[x_1,\ldots,x_n]$, say that 
$M$ a $\ZZ$-graded $A$-module has {\it componentwise linear resolution}
if for each $j \in \ZZ$ there is an $A$-linear resolution for
the pure $A$-submodule $M_{\langle j \rangle}$ generated by
the elements $M_j$ in $M$ of degree $j$.   

When $M$ is pure (generated in a single degree) this turns out to be equivalent 
to $M$ having linear resolution; see \cite[Lemmas 1,3]{HRW} or 
Proposition~\ref{prop:Yanagawa} below.
\end{definition}

For Stanley-Reisner ideals $I_\Delta$, componentwise
linearity turns out to correspond to a relaxation of the notion
of Cohen-Macaulayness called sequential Cohen-Macaulayness, 
introduced by Stanley \cite[\S III.2.9]{St}.
We recall here a reformulation of this definition due to 
Wachs \cite[Theorem 1.5]{Wa}, based on work of Duval \cite[\S 2]{Duv}.

\begin{definition}
\label{defn:sequential-Cohen-Macaulay}
\rm \ \\
For an abstract simplicial complex $\Delta$, the
{\it $j^{th}$ sequential layer} $\Delta^{\langle j \rangle}$ is
the subcomplex of $\Delta$ generated by all faces
of $\Delta$ of dimension at least $j$.

Say $\Delta$ is {\it sequentially acyclic over k} if
$\tilde{H}_i(\Delta^{\langle j \rangle};k)=0$ for $-1 \leq i < j \leq \dim\Delta$.

Say $\Delta$ is {\it sequentially Cohen-Macaulay over k} if
every face $F$ of $\Delta$ has $\link_\Delta(F)$ 
sequentially acyclic over $k$.
\end{definition}

For {\it pure} simplicial complexes $\Delta$, that is, those
with all facets (= maximal faces) of the same dimension,
it turns out that sequential Cohen-Macaulayness is equivalent to Cohen-Macaulayness. 
Since the ideal $I_\Delta$ is pure as an $A$-module if and only if
$\Delta^\vee$ is pure, the next theorem gives the appropriate nonpure generalization
of Theorem~\ref{thm:EagonR}.  It is the third result that we wish to unify.

\begin{theorem} (Herzog and Hibi \cite[Theorem 2.1]{HH}, Herzog, Reiner and Welker \cite[Theorem 9]{HRW})
\label{thm:EHHRW}

The Stanley-Reisner ideal $I_\Delta$ in $A=k[x_1,\ldots,x_n]$
has componentwise linear resolution as an $A$-module if and only if 
the Alexander dual complex $\Delta^\vee$ is sequentially Cohen-Macaulay over $k$.
\end{theorem}

\subsection{A unified setting and the main result}
\label{sec:unification}

We explain here how Theorem~\ref{thm:main} 
will unify the previous three results and settings, that is,
\begin{enumerate}
\item[(i)] Theorem~\ref{thm:Polo-Woodcock} on incidence algebras of posets,
\item[(ii)] Theorem~\ref{thm:PRS} on affine semigroup rings, and
\item[(iii)] Theorem~\ref{thm:EHHRW} on squarefree monomial ideals in polynomial algebras.
\end{enumerate}

Although at first glance, (ii) and (iii) look somewhat different,
the differences are superficial.  One key point is the following
observation  made independently by Yanagawa \cite{Yan} and R\"{o}mer \cite{Rom},
which we rephrase here in our language. 
See also the work of Iyengar and R\"{o}mer \cite[Theorem 5.6]{IR} for further extensions.

\begin{proposition}(\cite[Proposition 4.9]{Yan})
\label{prop:Yanagawa}

A finitely generated $\ZZ$-graded module $M$ over 
$A=k[x_1,\ldots,x_n]$ is componentwise linear if and only if 
$\gr_\mm M$ is a Koszul $A$-module, that is, if and only if $M$ itself is a nongraded Koszul
$A$-module.  
\end{proposition}
A second key point is that the polynomial algebra $A=k[x_1,\ldots,x_n]$ is
an affine semigroup ring, associated to the semigroup $\Lambda=\NN^n$.  
Thus we should consider nongraded Koszulness not only
for affine semigroup rings as algebras, but also for monomial ideals within them
as modules.

To unify (i) with (ii), (iii), 
one should think of an affine semigroup ring $k\Lambda$ as a certain kind of
{\it reduced incidence algebra} $k[P]_\red$, associated with the poset $P=\Lambda$
and the equivalence relation $\sim$ on the intervals $\Int(P)$
for which 
\begin{equation}
\label{semigroup-equivalence}
[\lambda,\mu] \sim [\lambda',\mu'] \quad \text{ when } \quad  \mu-\lambda=\mu'-\lambda'.
\end{equation}
We generalize this construction as follows.

\begin{definition}
\label{defn:reduced-incidence-algebra}
(cf. Doubilet, Rota and Stanley \cite[Definition 4.1, Proposition 4.3]{DRS})
\rm \ \\
Let $P$ be a poset in which every interval $[x,y]$ is finite,
and assume one has an equivalence relation on $\Int(P)$
which is {\it order-compatible} in the sense of 
Axiom \ref{ax:order-compatibility} below.
Define $k[P]_\red$ to be the $k$-vector space 
having basis $\{ \xi_{\widetilde{[x,y]}} \}$ indexed by 
the equivalence classes $\Int(P)/\!\!\sim$,
with multiplication defined $k$-bilinearly via 
\begin{equation}
\label{reduced-incidence-multiplication}
\xi_{\widetilde{[x,y]}} \cdot \xi_{\widetilde{[z,w]}} 
= \xi_{\widetilde{[x',w']}}
\end{equation}
if there exist $x' \leq y' \leq w'$ in $P$ with
$$
\begin{aligned}[]
[x',y'] &\sim [x,y] \\
[y',w'] &\sim [z,w],
\end{aligned}
$$
and zero otherwise.
\end{definition}

\begin{remark} (On incidence algebras for infinite posets) \rm \ \\ 
We should point out that in \cite{DRS}, elements of the reduced incidence
algebra are defined to be functionals on the equivalence classes $\Int(P)/\!\!\sim$ of 
intervals in $P$, and allowed to have infinite support.  However, our definition 
corresponds to the subalgebra of functionals which are supported on only finitely many
equivalence classes.
\end{remark}

Without further assumptions on $P$ and $\sim$, 
the multiplication \eqref{reduced-incidence-multiplication} may be badly behaved.  
In Section~\ref{sec:axioms} we will impose one-by-one four axioms
(Axioms \ref{ax:order-compatibility}, \ref{ax:invariance}, 
\ref{ax:finiteness}, \ref{ax:concatenation})
on $(P,\sim)$, and explain the role that 
each plays.  We avoid discussing these axioms fully here, as they are
slightly technical, and concentrate instead on their implications.
However, it will be an easy exercise to check that all four axioms
are satisfied both by
\begin{enumerate}
\item[$\bullet$]
finite posets $P$ with the trivial equivalence relation $\sim$ 
on $\Int(P)$, in which case $k[P]_\red=k[P]$ is the usual incidence algebra, and
\item[$\bullet$]
by affine semigroups $\Lambda$ considered as a poset $P$ with
the equivalence relation \eqref{semigroup-equivalence},
in which case $k[P]_\red = k\Lambda$ is the affine semigroup ring $k\Lambda$.
\end{enumerate}

\noindent
For considering Koszulness, decompose
\begin{equation}
\label{nongraded-decomposition}
A=k[P]_\red=A_0 \oplus I
\end{equation}
in which $A_0$ and $I$, respectively, are the $k$-spans of those
elements $\xi_{\widetilde{[x,y]}}$ for which $x=y$ and $x < y$, respectively.
It is not hard to see that for $(P,\sim)$ obeying these axioms, the
rings $A=k[P]_\red$ and $\gr_I A$ are isomorphic as $\NN$-graded rings 
if and only if $P$ is a graded poset.
The axioms will also imply that $A_0$ is a semisimple subalgebra, so that we
can consider Koszulness.

Lastly, we must define the notion of a monomial ideal of $k[P]_\red$.

\begin{definition}
\label{defn:monomial-ideal}
\rm \ \\
Let $P$ be a poset and $\sim$ an equivalence relation on $\Int(P)$, satisfying 
Axioms \ref{ax:order-compatibility}, \ref{ax:invariance}, 
\ref{ax:finiteness}, \ref{ax:concatenation}.

In the reduced incidence algebra $A=k[P]_\red$, call a $k$-basis element  $\xi_\alpha$ for
$\alpha \in \Int(P)/\!\!\sim$ a {\em monomial}. Say that a subspace $J$ of the two-sided ideal $I$
appearing in the decomposition \eqref{nongraded-decomposition} is a {\em monomial (left, or right, or two-sided) ideal}
if  it is a (left, or right, or two-sided) ideal in $A$ generated by a (possibly infinite) set of monomials in $A$.

Hence, such a subspace $J$ is an ($A-A_0$-, or $A_0-A$-, or $A-A$-)bi-submodule of $A$.
For example, the ideal $I$ itself is a monomial ideal which contains
all other monomial ideals $J$.

We will abuse notation slightly in denoting by $J$ both the ideal in $k[P]_\red$
and the subset of elements $\alpha$ in $\Int(P)/\!\!\sim$ that
index elements $\xi_\alpha$ in $J$.
Our main result is phrased in terms of
certain subposets $(x,y)|_J$ of the open intervals $(x,y)$ in $P$:
\begin{equation}
\label{J-restricted-interval}
(x,y)|_J :=\{z \in (x,y): \widetilde{[x,z]} \in J\}.
\end{equation}
\end{definition}

We can now state our main result.

\begin{theorem}
\label{thm:main}{}

Let $P$ be a poset and $\sim $ an equivalence relation on $\Int(P)$, satisfying Axioms
\ref{ax:order-compatibility}, \ref{ax:invariance}, \ref{ax:finiteness}, \ref{ax:concatenation}.  
Let $J$ be a monomial right ideal in $k[P]_\red$.

Then $J$ is a nongraded Koszul right $k[P]_\red$-module
if and only if for every equivalence class $\widetilde{[x,y]}$ in $J$
the poset $(x,y)|_J$ is sequentially acyclic over $k$.
\end{theorem}

The case where the monomial ideal $J$ is chosen to be the ideal $I$ 
from \eqref{nongraded-decomposition} easily gives the following
corollary (see Section~\ref{sec:deducing-main-corollary}), 
which immediately implies Theorems~\ref{thm:Polo-Woodcock} and \ref{thm:PRS}.

\begin{corollary}
\label{cor:main-algebra}

Let $P$ be a poset and $\sim $ an equivalence relation on $\Int(P)$, satisfying Axioms
\ref{ax:order-compatibility}, \ref{ax:invariance}, \ref{ax:finiteness}, \ref{ax:concatenation}.  

Then $k[P]_\red$ is a nongraded Koszul ring 
if and only if every open interval $(x,y)$ in $P$
is sequentially Cohen-Macaulay over $k$.
\end{corollary}

We explain here how Theorem~\ref{thm:main} also captures Theorem~\ref{thm:EHHRW},
leaving details to Section~\ref{sec:deducing-EHHRW}.
Consider $\Lambda=\NN^n$ as an affine semigroup, 
with $P=\Lambda$ and $\sim$ the equivalence relation \eqref{semigroup-equivalence}.
Here $A=k[P]_\red=k[x_1,\ldots,x_n]$ is $\NN$-graded,
so that $\gr_I A= A$.  For any Stanley-Reisner ideal $J=I_\Delta$ 
in $A$, one needs to check that sequential Cohen-Macaulayness
of $\Delta^\vee$ is equivalent to
the sequential acyclicity of $(x,y)|_J$ for
all $\widetilde{[x,y]}$ in $J$.
Comparing with the Definition~\ref{defn:sequential-Cohen-Macaulay},
one needs to check the equivalence between
sequential acyclicity of all
\begin{enumerate}
\item[$\bullet$]
links of faces in $\Delta^\vee$, and
\item[$\bullet$]
subposets $(x,y)|_J$ for $\widetilde{[x,y]}$ in $J$.
\end{enumerate}
It will turn out that all sequential layers of $(x,y)|_J$ are contractible
unless 
$
[x,y] \sim [1,\prod_{i \not\in F} x_i]
$
for some face $F$ of $\Delta^\vee$.  And in the latter case,
$\Delta( (x,y)|_J )$ turns out to be the barycentric subdivision
of the link of the face $F$ in $\Delta^\vee$, so that the
sequential acyclicity for the two is equivalent.

\vskip .2in

The remainder of the paper is structured as follows.

Section~\ref{sec:axioms} gives the four axioms to be
imposed on a poset $P$ and equivalence relation $\sim$
on $\Int(P)$ before considering the reduced incidence
algebra $k[P]_\red$.

Section \ref{sec:tor} is the crux of the proof of the main
result.  One shows via an argument with the bar resolution,
going back essentially to the work of Laudal and Sletsj{\o}e \cite{LS},
how to compute 
$\Tor^{\gr_I A}(\gr_I J,A_0)$ where $J$ is any monomial right ideal
of $A=k[P]_\red$,
in terms of the homology of the subposets
$(x,y)|_J$; see Theorem~\ref{thm:tor-computation}.

Section~\ref{sec:koszulity} uses 
Theorem~\ref{thm:tor-computation} to prove Theorem~\ref{thm:main}.
A key observation is a reformulation of the 
sequential Cohen-Macaulay property (Proposition~\ref{prop:new-seq-C-M-defn}),
tailored to this purpose. 

Section~\ref{sec:numerology} pauses to discuss in this setting
how standard numerology for Koszul rings interacts
with standard numerology for sequentially Cohen-Macaulay complexes.

Section \ref{sec:quadraticity} uses Theorem~\ref{thm:tor-computation}
to characterize combinatorially when the rings $k[P]_\red$ are quadratic.

Section \ref{sec:subgroups}  discusses two (essentially) known
results that characterize the Cohen-Macaulay and sequential Cohen-Macaulay
properties for the lattice of subgroups $L(G)$ of a finite group $G$.
A result of Bj\"orner \cite[Theorem 3.3]{Bj} characterizes Cohen-Macaulayness, 
while a result of Shareshian \cite[Theorem 1.4]{Sha} (essentially) characterizes
sequential Cohen-Macaulayness, after we add some very trivial sharpening.
We present these last results here because it would be desirable to unify 
them with Theorem~\ref{thm:main}.

\section{The axioms}
\label{sec:axioms}

Let $P$ be a poset and $\sim$ an equivalence relation on its intervals $\Int(P)$.
The following are the axioms to be imposed on $(P,\sim)$ prior to proving our main results
on the reduced incidence algebra $k[P]_\red$ from
Definition~\ref{defn:reduced-incidence-algebra}.
We list them all, then explain the role of each in the behavior of $k[P]_\red$.

\vskip .2in

\begin{axiom}(order-compatibility)
\label{ax:order-compatibility}
\newline
\noindent
If $[x,y] \sim [x',y']$ and $[y,z] \sim [y',z']$, then $[x,z] \sim [x',z']$.
\end{axiom}

\begin{axiom}(invariance)
\label{ax:invariance}
\newline
\noindent
The (lower, upper) interval mappings
\begin{equation}
\label{int-mapping}
\begin{array}{rrcl}
\intervals_{[x,y]}    
              &: [x,y] &\longrightarrow & \Int(P)/\!\!\sim \,\,\, \times \,\,\, \Int(P)/\!\!\sim\\
              & & \\
              &  z     &\longmapsto
                          &(\widetilde{[x,z]},\widetilde{[z,y]}) \\
\end{array}
\end{equation}
have the property that whenever $[x,y] \sim [x',y']$, 
\begin{enumerate}
\item[(a)] there exists a map $\tau: [x,y] \rightarrow [x',y']$
that commutes with $\intervals_{[x,y]}, \intervals_{[x',y']}$:
$$
\intervals_{[x',y']} \,\, \circ \,\, \tau = \intervals_{[x,y]},
$$
\item[(b)] and such a map $\tau$ is unique.
\end{enumerate}
\end{axiom}

\begin{axiom}(finiteness)
\label{ax:finiteness}
\newline
\noindent
The equivalence relation $\sim$ on $P$ defined by $x \sim y$ if $[x,x] \sim [y,y]$
has only finitely many equivalence classes $P/\!\!\sim$.
\end{axiom}

\begin{axiom}(concatenation)
\label{ax:concatenation}
\newline
\noindent
If $x \leq y_1$ and $y_2 \leq z$ in $P$ with $y_1 \sim y_2$, then there exist $x' \leq y' \leq z'$ in $P$
with
$$
\begin{aligned}[]
[x',y'] &\sim [x,y_1] \\
[y',z'] &\sim [y_2,z].
\end{aligned}
$$
\end{axiom}

\vskip .2in

Before explaining the significance of the axioms, we introduce a convenient
shorthand notation:  write
$$
x_0 \overset{\alpha_1}{\leq} x_1 \overset{\alpha_2}{\leq} 
       \cdots \overset{\alpha_t}{\leq} x_t
$$
to mean that $x_{i-1} \leq x_i$ in $P$ and $[x_{i-1},x_i] \sim \alpha_i$, that is, 
$\alpha_i=\widetilde{[x_{i-1},x_i]}$ in  $\Int(P)/\!\!\sim$. 
With this shorthand the definition of the multiplication in $k[P]_\red$ is 
$$
\xi_\alpha \cdot \xi_\beta = 
\begin{cases}
\xi_\gamma &\text{ if there exists }x \overset{\alpha}{\leq} y \overset{\beta}{\leq} z
\text{ with }x \overset{\gamma}{\leq} z,\\
0 & \text{ otherwise.}
\end{cases}
$$

\subsection{Role of Axiom~\ref{ax:order-compatibility}}
Axiom~\ref{ax:order-compatibility} is exactly what one needs for the 
multiplication in $k[P]_\red$ given
in Definition~\ref{defn:reduced-incidence-algebra} to be well-defined.

\subsection{Role of Axiom~\ref{ax:invariance}}
Axiom~\ref{ax:invariance} part (a) is the most convenient hypothesis
for concluding {\it associativity} of the
multiplication in $k[P]_\red$ given
in Definition~\ref{defn:reduced-incidence-algebra}.

\begin{proposition}
\label{prop:associativity}
Axioms~\ref{ax:order-compatibility} and \ref{ax:invariance} part (a)
imply that $k[P]_\red$ is an associative algebra.
\end{proposition}
\begin{proof}
By bilinearity, it suffices to check for any three equivalence
classes $\alpha, \beta, \gamma$ in $\Int(P)/\!\!\sim$ that
$$
( \xi_\alpha \xi_\beta ) \xi_\gamma
=  \xi_\alpha ( \xi_\beta  \xi_\gamma ).
$$
If both sides are zero, there is nothing to show.  Assume that the
left side is nonzero;  the argument will be similar if one assumes that
the right side is nonzero.  

This implies $\xi_\alpha \xi_\beta=:\xi_\delta$ is nonzero, and hence there exist
$x \overset{\alpha}{\leq }y \overset{\beta}{\leq} z$ in $P$ with 
$x \overset{\delta}{\leq} z$.
Since 
$
( \xi_\alpha \xi_\beta ) \xi_\gamma = \xi_\delta \xi_\gamma = :\xi_\epsilon
$ 
is nonzero, there must also exist
$x' \overset{\delta}{\leq} z' \overset{\gamma}{\leq} w'$ in $P$ with
$x' \overset{\epsilon}{\leq} w'$.

Now use the map $\tau:[x,z] \rightarrow [x',z']$
provided by Axiom~\ref{ax:invariance} part (a) to produce $y':=\tau(y)$
with 
$
x' \overset{\alpha}{\leq} y' \overset{\beta}{\leq} z' \overset{\gamma}{\leq} w'.
$
in $P$.  However, the existence of this (multi-)chain
also shows that $\xi_\alpha (\xi_\beta \xi_\gamma)$ is
nonzero, and equal to $\xi_\epsilon$.
\end{proof}

Axiom~\ref{ax:invariance} part (b) plays a different role, in that it
implies a certain {\it internal} rigidity for every interval $[x,y]$
(see Proposition~\ref{prop:rigidity}(i) below), as well as showing that
the maps $\tau$ are consistent whenever their domains overlap
(see Proposition~\ref{prop:rigidity}(iii) below), and 
that they provide a coherent family of isomorphisms
between $\sim$-equivalent intervals (see 
Proposition~\ref{prop:rigidity}(ii) below).
This last property explains why it makes sense to consider the smaller
object $k[P]_\red$ instead
of all of $k[P]$ whenever the relation $\sim$ is nontrivial.

\begin{proposition}
\label{prop:rigidity}
Axioms~\ref{ax:order-compatibility} and \ref{ax:invariance} 
imply the following.

\begin{enumerate}
\item[(i)]
The map
$
\intervals_{[x,y]} : [x,y] \longrightarrow  \Int(P)/\!\!\sim \,\,\, \times \,\,\, \Int(P)/\!\!\sim
$
defined in \eqref{int-mapping} is always injective.

\item[(ii)]
When $[x,y] \sim [x',y']$, the
two maps provided by Axiom~\ref{ax:invariance}
$$
\begin{aligned}[]
[x,y] &\overset{\tau}{\longrightarrow} [x',y'] \\
[x',y'] &\overset{\tau'}{\longrightarrow} [x,y] \\
\end{aligned}
$$
are mutually inverse poset isomorphisms.

\item[(iii)]
The maps $\tau$ are consistent in the sense that 
whenever $[x,z] \sim [x',z']$ and $\tau:  [x,z] \rightarrow [x',z']$
is the unique map in Axiom~\ref{ax:invariance}, then
for any $y$ in $[x,z]$, the two other maps
\begin{equation}
\label{coherent-restrictions}
\begin{aligned}
\tau_{[x,y]}: & [x,y] \rightarrow [x',\tau(y)] \\
\tau_{[y,z]}: & [y,z] \rightarrow [\tau(y),z']
\end{aligned}
\end{equation}
coincide with the restrictions of $\tau$ to 
the domains $[x,y]$ and $[y,z]$, respectively.

\end{enumerate}
\end{proposition}

\begin{proof}
For (i), apply both parts of Axiom~\ref{ax:invariance} to the
trivial equivalence $[x,y] \sim [x,y]$;  the map
$\intervals_{[x,y]}$ out of $[x,y]$ is injective if and only if
the identity map is the unique function from $[x,y]$ to itself that
commutes with $\intervals_{[x,y]}$.

For (ii), one first argues that the maps $\tau, \tau'$
are mutually inverse as {\it set} maps in the customary way for 
maps satisfying universal properties:
the composites $\tau'\circ \tau$ and $\tau \circ \tau'$
commute with the maps $\intervals_{[x,y]},\intervals_{[x',y']}$, 
out of $[x,y]$ and $[x',y']$, and hence must coincide with the identity maps on $[x,y]$ and $[x',y']$
by the uniqueness in Axiom~\ref{ax:invariance} part (b).

To show that $\tau: [x,y] \rightarrow [x',y']$ is order-preserving,
we assume $x \leq w \leq z \leq y$ and then we check that $\tau(w) \leq \tau(z)$.
Since $[x,z] \sim [x',\tau(z)]$, applying Axiom~\ref{ax:invariance} part (a) 
provides us with an element $w'$ for which 
\begin{equation}
\label{w'-properties}
\begin{aligned}[]
 [x,w] & \sim [x',w'],\\
 [w,z] & \sim [w',\tau(z)].\\
\end{aligned}
\end{equation}
Since $[z,y]\sim [\tau(z), y']$, combining this with the second equivalence
in \eqref{w'-properties} and using Axiom~\ref{ax:order-compatibility},
one concludes that
\begin{equation}
\label{another-w'-property}
[w,y] \sim [w',y'].
\end{equation}
Combining the first equivalence in \eqref{w'-properties} with \eqref{another-w'-property},
one obtains
$$
\begin{aligned}[]
[x',w'] & \sim [x,w] ( \sim [x',\tau(w)] ), \\
[w',y'] & \sim [w,y] ( \sim [\tau(w),y'] ), \\
\end{aligned}
$$
and hence $w'$, $\tau(w)$ have the same image under 
$\intervals_{[x',y']}:[x',y'] \rightarrow (\Int(P)/\!\!\sim\!\!~)^2$.  
Thus $w'=\tau(w)$ by the injectivity of $\intervals_{[x',y']}$
already proven in part (i). We conclude that $\tau(w) \leq \tau(z)$.

For (iii), note that since $\tau:[x,z] \rightarrow [x',z']$ is order-preserving by part (ii),
its restrictions to the domains $[x,y]$ and $[y,z]$ do give maps 
$$
\begin{aligned}
\Res^{[x,z]}_{[x,y]}(\tau)
: & [x,y] \rightarrow [x',\tau(y)], \\
\Res^{[x,z]}_{[y,z]}(\tau)
: & [y,z] \rightarrow [\tau(y),z'].
\end{aligned}
$$
We wish to show that these restrictions of $\tau$ coincide with the maps 
$\tau_{[x,y]}, \tau_{[y,z]}$ in \eqref{coherent-restrictions}.
We give the argument here for why $\Res^{[x,z]}_{[x,y]}(\tau) = \tau_{[x,y]}$;
the other equality is similar.   

By Axiom~\ref{ax:invariance}(b) applied to $\tau$, it suffices to
show that for a typical element $u$ in $[x,y]$, one has equivalences
$[x',\tau_{[x,y]}(u)] \sim [x', \tau(u)]$ and $[\tau_{[x,y]}(u),z'] \sim [\tau(u),z']$.
The first equivalence holds because both intervals $[x',\tau_{[x,y]}(u)]$ and 
$[x', \tau(u)]$ are equivalent to $[x,u]$, by the defining properties of $\tau_{[x,y]}$
and $\tau$.  For the second of these equivalences, first apply 
Axiom~\ref{ax:order-compatibility} to
$$
\begin{aligned}[]
[u,y] & \sim [\tau_{[x,y]}(u),\tau(y)]  \textrm{ and}\\
[y,z] & \sim [\tau(y),z'],  \\
\end{aligned}
$$ 
obtaining $[u,z] \sim [\tau_{[x,y]}(u),z']$.
Then one has 
$$
[\tau_{[x,y]}(u),z'] \sim [u,z] \sim [\tau(u),z'],
$$
as desired.
\end{proof}

We remark on a few consequences of Proposition~\ref{prop:rigidity} for
defining combinatorial invariants of the equivalence classes $\alpha=\widetilde{[x,y]}$
in $\Int(P)/\!\!\sim$ in terms of their representative intervals $[x,y]$.

\begin{definition}
\rm \ \\
Define the equivalence relation $\sim$
on $P$ by $x \sim y$ if $[x,x] \sim [y,y]$, and write $P/\!\!\sim$
for the collection of equivalence classes $\{\tilde{x}\}=\{\widetilde{[x,x]}\}$.
\end{definition}

Under the assumptions of Proposition~\ref{prop:rigidity}, part (i) says that the {\it source}
and {\it target} maps
$$
\begin{array}{rrcl}
s,t: &\Int(P)/\!\!\sim &\longrightarrow &P/\!\!\sim \\
     &\alpha:=\widetilde{[x,y]}&\overset{s}{\longmapsto} &s(\alpha):=\tilde{x} \\
     &\alpha:=\widetilde{[x,y]}&\overset{t}{\longmapsto} &t(\alpha):=\tilde{y}
\end{array}
$$
are well-defined.

\begin{definition}
\rm \ \\
Define the {\it length} $\ell[x,y]$ of an interval in $P$ to be
the length $\ell$ of the longest chain $x=x_0 < x_1 < \cdots < x_\ell =y$.
\end{definition}

Under the assumptions of Proposition~\ref{prop:rigidity}, part(ii) implies
that one can define the length $\ell(\alpha):=\ell[x,y]$ 
for an equivalence class $\alpha=\widetilde{[x,y]}.$

Having defined the notion of length, one can check that,
assuming both Axioms~\ref{ax:order-compatibility} and \ref{ax:invariance},
one has the following
description of the associated graded ring $\gr_I A$ when $A=k[P]_\red$:
it has $k$-basis $\{ \bar{\xi}_{\alpha} \}$ indexed by $\alpha$ in $\Int(P)/\!\!\sim$,
with multiplication defined $k$-bilinearly by
\begin{equation}
\label{gr-multiplication}
\bar{\xi}_{\widetilde{[x,y]}} \cdot \bar{\xi}_{\widetilde{[z,w]}} 
= \bar{\xi}_{\widetilde{[x',w']}}
\end{equation}
if there exist $x' \leq y' \leq w'$ in $P$ with
$$
\begin{aligned}[]
[x',y'] &\sim [x,y] \\
[y',w'] &\sim [z,w] \\
\ell[x',w']&=\ell[x',y']+\ell[y',w']
\end{aligned}
$$
and $\bar{\xi}_{\widetilde{[x,y]}} \cdot \bar{\xi}_{\widetilde{[z,w]}}=0$ otherwise. 
Furthermore, the length function $\ell$ describes the
$\NN$-graded components of $\gr_I A$:
the $d^{th}$ homogeneous component 
$(\gr_I A)_d = I^d/I^{d+1}$ is the $k$-span of
$\{ \bar{\xi_\alpha}: \ell(\alpha)=d \}$.

\subsection{Role of Axiom~\ref{ax:finiteness}}
Axiom~\ref{ax:finiteness} is imposed so that $A:=k[P]_\red$
is a ring with unit, and so that
when one decomposes $A=A_0 \oplus I$ as in \eqref{nongraded-decomposition},
the $k$-subalgebra $A_0$ is finite-dimensional over $k$.
We can be more precise about its consequences.
Let $R:=\gr_I A$, so that $R_0=A/I=A_0$.

For each $\alpha$ in $\Int(P)/\!\!\sim$, the
one-dimensional $k$-vector spaces 
$$
\begin{aligned}
A_\alpha:=k\xi_\alpha & \subset A \\
R_\alpha:=k\bar{\xi}_\alpha & \subset R \\
\end{aligned}
$$
are $A_0-A_0$-sub-bimodules of $A, R$.
Note that when $\alpha=\tilde{x}$ in $P/\!\!\sim$, the
element $\xi_{\tilde{x}}=\bar{\xi}_{\tilde{x}}$ 
is idempotent ($\xi_{\tilde{x}}^2=\xi_{\tilde{x}}$), and hence 
$k_{\tilde{x}}:=A_{\tilde{x}}=k\xi_{\tilde{x}}$ 
is a subalgebra of $A_0, A,$ and $R$,
isomorphic to the field $k$.

\begin{proposition}
\label{prop:finiteness}
Assume the pair $(P,\sim)$ satisfies Axioms~\ref{ax:order-compatibility},
\ref{ax:invariance}, and \ref{ax:finiteness}.
Then the elements
$
\{ \xi_{\tilde{x}} : \tilde{x} \in P/\!\!\sim \}
$
form a (finite) complete system of mutually orthogonal idempotents
for the rings $A_0, A, R$.  In particular, the ring
$A_0$ is the direct product of the (finitely) many fields $k_{\tilde{x}}$,
and hence a commutative semisimple finite-dimensional $k$-algebra.

Furthermore, $A, R$ have the following $A_0-A_0$-bimodule
decompositions:
$$
\begin{aligned}
A &= \bigoplus_{(\tilde{x},\tilde{y}) \in P/\! \sim \times P/ \!\sim} 
      \xi_{\tilde{x}} A \xi_{\tilde{y}} \\
R &= \bigoplus_{(\tilde{x},\tilde{y}) \in P/\! \sim \times P/ \!\sim} 
      \xi_{\tilde{x}} R \xi_{\tilde{y}}\\
\text{ where }& \\
\xi_{\tilde{x}} A \xi_{\tilde{y}}
&= \bigoplus_{\substack{\alpha \in \Int(P)/\! \sim:\\
                s(\alpha) =\tilde{x}\\ t(\alpha) = \tilde{y}}}  A_{\alpha} \\
\xi_{\tilde{x}} R \xi_{\tilde{y}}
&= \bigoplus_{\substack{\alpha \in \Int(P)/\! \sim:\\
                s(\alpha) = \tilde{x}\\ t(\alpha) = \tilde{y}}}  R_{\alpha} 
\end{aligned}
$$
\end{proposition}

\begin{proof}
We first show Axioms~\ref{ax:order-compatibility},
\ref{ax:invariance}, \ref{ax:finiteness} imply that 
$
\{ \xi_{\tilde{x}} : \tilde{x} \in P/\!\!\sim \}
$
are mutually orthogonal idempotent elements in the ring  $A$ (hence also in $A_0$). Namely, that
\begin{equation}
\label{eq-orthogonal-idempotents}
\xi_{\tilde{x}}\cdot \xi_{\tilde{y}}=\delta_{\tilde{x},\tilde{y}} \xi_{\tilde{x}}
\end{equation}
holds, where $\delta_{\tilde{x},\tilde{y}} $ is the Kronecker delta.
Indeed, if $\xi_{\tilde{x}}\cdot \xi_{\tilde{y}} \neq 0$, there exists $u\leq v\leq w$ such that
$[u,v]\sim [x,x]$ and $[v,w] \sim [y,y]$ and $\xi_{\tilde{x}}\cdot \xi_{\tilde{y}}=\xi_{\widetilde{[u,w]}}$.
By Proposition~\ref{prop:rigidity}, $\ell([u,v])=\ell([x,x])=0$ and $\ell([v,w])=\ell([y,y])=0$. 
This gives $u=v=w$ and $x\sim y$.

Note that given $m=\sum_{\alpha \in \Int(P)/\! \sim} c_{\alpha} \xi_\alpha$ in $A$, one can rewrite it as
\begin{equation}
\label{eq-an-element-in-A}
m= \sum_{\alpha \in \Int(P)/ \!\sim} c_{\alpha} (\xi_{s(\alpha)} \cdot \xi_\alpha \cdot \xi_{t(\alpha)}).
\end{equation}
Use this observation, Axiom \ref{ax:finiteness} and \eqref{eq-orthogonal-idempotents} to conclude that
the multiplicative identity element in $A$ is $1=\sum_{\tilde{x} \in P/\!\sim} \xi_{\tilde{x}}$.

Equation \eqref{eq-an-element-in-A} may be used again to derive the decomposition for $A$ in the second 
part of the conclusion. 

Similar arguments are required to prove the assertions about the ring $R$.
\end{proof}

\subsection{Role of Axiom~\ref{ax:concatenation}}
We will be working with tensor products 
$$
A^{\otimes_{A_0}i}:=
\underbrace{A \otimes_{A_0} A \otimes_{A_0} \cdots \otimes_{A_0} A}_{i \text{ times}}
$$
considered as $A_0-A_0$-bimodules, and similarly for
$R^{\otimes_{A_0}i}$ where $R=\gr_I A$.
Axiom~\ref{ax:concatenation} is imposed so as to make 
their $A_0-A_0$-bimodule decomposition 
indexed by $\Int(P)/\!\!\sim$, similar to the case
$i=1$ described in Proposition~\ref{prop:finiteness}.

For notational purposes we are fixing once and for all for every
equivalence class $\alpha$ in $\Int(P)/\!\!\sim$ a representative interval
$[x(\alpha),y(\alpha)]$, i.e.
\begin{equation}
\label{eqn:xalpha-yalpha}
\alpha=\widetilde{[x(\alpha),y(\alpha)]}
\end{equation}

\begin{proposition}
\label{prop:tensor-bases}
Assume the pair $(P,\sim)$ satisfies Axioms~\ref{ax:order-compatibility}, 
\ref{ax:invariance}, \ref{ax:finiteness}, \ref{ax:concatenation}.
Then
$$
\begin{aligned}
A^{\otimes_{A_0}i}
&= \bigoplus_{(\tilde{x},\tilde{y}) \in P/\! \sim \times P/\! \sim} 
      \xi_{\tilde{x}} A^{\otimes_{A_0}i} \xi_{\tilde{y}}, \\
R^{\otimes_{A_0}i}
&= \bigoplus_{(\tilde{x},\tilde{y}) \in P/\! \sim \times P/\! \sim} 
      \xi_{\tilde{x}} R^{\otimes_{A_0}i} \xi_{\tilde{y}},\\
\text{ where }& \\
\xi_{\tilde{x}} A^{\otimes_{A_0}i} \xi_{\tilde{y}}
&= \bigoplus_{\substack{\alpha \in \Int(P)/\! \sim:\\
                s(\alpha) =\tilde{x}\\ t(\alpha) = \tilde{y}}}  
                     (A^{\otimes_{A_0}i})_\alpha ,\\
\xi_{\tilde{x}} R^{\otimes_{A_0}i} \xi_{\tilde{y}}
&= \bigoplus_{\substack{\alpha \in \Int(P)/\! \sim:\\
                s(\alpha) = \tilde{x}\\ t(\alpha) = \tilde{y}}}  
                     (R^{\otimes_{A_0}i})_\alpha ,
\end{aligned}
$$
in which 
$(A^{\otimes_{A_0}i})_\alpha$ and
$(R^{\otimes_{A_0}i})_\alpha$
both have $k$-bases
$$
\begin{aligned}
&\xi_{\widetilde{[x_0,x_1]}} \otimes 
\xi_{\widetilde{[x_1,x_2]}} \otimes 
\cdots \otimes
\xi_{\widetilde{[x_{i-1},x_i]}} \\
&\bar{\xi}_{\widetilde{[x_0,x_1]}} \otimes 
\bar{\xi}_{\widetilde{[x_1,x_2]}} \otimes 
\cdots \otimes
\bar{\xi}_{\widetilde{[x_{i-1},x_i]}} 
\end{aligned}
$$
indexed by length $i$ (multi-)chains
$$
x(\alpha) =: x_0 \leq x_1 \leq \cdots \leq x_{i-1} \leq x_i :=y(\alpha).
$$\end{proposition}

\begin{proof}
We give the proof for $A$; the argument for $R$ is
essentially the same. Using the decomposition of $A_0$ as a direct
sum of the fields $k_{\tilde{x}}$, 
the $A_0$-multilinearity of the tensor products,
and the $A_0-A_0$-structure of $A$
described in Proposition~\ref{prop:finiteness},
it is easy to see that
$\xi_{\tilde{x}} A^{\otimes_{A_0}i} \xi_{\tilde{y}}$
has $k$-basis
$$
\{\xi_{\alpha_1} \otimes \cdots \otimes \xi_{\alpha_i} \}
$$
indexed by sequences $(\alpha_1,\ldots,\alpha_i)$ in $(\Int(P)/\!\!\sim)$
with
$t(\alpha_{j-1})=s(\alpha_j)$
for $j=1,2,\ldots,i+1$, assuming the
conventions that
$\alpha_0:=\tilde{x}, \alpha_{i+1}:=\tilde{y}$.
By iterating Axiom~\ref{ax:concatenation}, one sees that
for each such sequence $(\alpha_1,\ldots,\alpha_i)$, there exists
$$
x_0 \overset{\alpha_1}{\leq} x_1 \overset{\alpha_2}{\leq} 
       \cdots \overset{\alpha_i}{\leq} x_i
$$ 
in $P$ with $\tilde{x_0}=\tilde{x}, \tilde{x_i}=\tilde{y}$.
Let $\alpha:=\widetilde{[x_0,x_i]}$.

We define a map $\phi_\alpha$ from the set  of all length $i$ multichains
having $x(\alpha)=:x_0$ and $x_i:=y(\alpha)$ as above to the set of all $i$-tuples
$(\alpha_1,\dots, \alpha_i)$ for which $t(\alpha_{j-1})=s(\alpha_j), 1 \leq j \leq i+1$
and $\xi_{\alpha_1}\cdots \xi_{\alpha_i}=\xi_\alpha$. More precisely, we let 
 $$
\phi_\alpha ( x_0 \overset{\alpha_1}{\leq} x_1 \overset{\alpha_2}{\leq} 
       \cdots \overset{\alpha_i}{\leq} x_i
  ) := (\alpha_1,\dots, \alpha_i).
$$

It only remains to show the bijectivity of the map $\phi_\alpha$.
Pick $(\alpha_1,\dots, \alpha_i)$ an element in the target set. Use Axiom~\ref{ax:concatenation}
to conclude the existence of some interval $[x_0', x_i']$ in the class of $\alpha$
containing an appropriate multichain. Apply Axiom~\ref{ax:invariance}(a) and conclude
that $[x(\alpha),y(\alpha)]$ also contains such a multichain. This proves the surjectivity.
To prove the injectivity, consider two multichains
$\{x_j\}_{0\leq j \leq i}$, $\{x_j'\}_{0\leq j \leq i}$ in $[x(\alpha),y(\alpha)]$ 
that have the same image under $\phi_\alpha$. Now use Proposition~\ref{prop:rigidity}(i)
to conclude that  $x_1=x'_1$ since
$$
\begin{aligned}
\xi_{\widetilde{[x_0,x_1]}} &= \xi_{\alpha_1}= \xi_{\widetilde{[x_0, x_1']}}, \\ 
\xi_{\widetilde{[x_1,x_i]}} &= \xi_{\alpha_2} \cdots \xi_{\alpha_i} = \xi_{\widetilde{[x_1', x_i]}}.
\end{aligned}
$$
Iterate this to conclude that $\phi_\alpha$ is injective.
\end{proof}

\section{Computing Tor}
\label{sec:tor}

In the following we assume that the pair $(P,\sim)$ satisfies Axioms~\ref{ax:order-compatibility}, 
\ref{ax:invariance}, \ref{ax:finiteness}, \ref{ax:concatenation}.
Let $A:=k[P]_\red$, and let $J$ be a monomial right ideal of $A$,
as defined in Definition~\ref{defn:monomial-ideal}.
Consider the associated graded ring 
$$
\begin{aligned}
R:=\gr_I A &= \bigoplus_{d \geq 0} I^d/I^{d+1} \\
&= \underbrace{A_0}_{\deg 0} \oplus \underbrace{I/I^2}_{\deg 1} \oplus 
      \underbrace{I^2/I^3}_{\deg 2} \oplus \cdots  
\end{aligned}
$$ 
and the graded right $R$-module associated to (the right $A$-module) $J$
$$
\begin{aligned}
\gr_I J &= \bigoplus_{d \geq 0} JI^d/JI^{d+1} \\
&= \underbrace{J/JI}_{\deg 0} \oplus \underbrace{JI/JI^2}_{\deg 1} \oplus 
      \underbrace{JI^2/JI^3}_{\deg 2} \oplus \cdots . 
\end{aligned}
$$
Then $\gr_I J$ has a $k$-basis $\{\hat{\xi}_{\alpha}\}$ indexed by  $\alpha$ in $J$,
and (right) $R$-module structure described 
as follows (to be compared with \eqref{gr-multiplication}):  for
$\alpha$ in $J$ and $\beta$ in $\Int(P)/\!\!\sim$,  
$$
\hat{\xi}_{ \alpha} \cdot \bar{\xi}_{ \beta} =
\hat{\xi}_{ \gamma}
$$
if there exists $\gamma$ in $J$ with
$$
\begin{aligned}
\xi_{\alpha} \cdot \xi_{\beta}&=\xi_{\gamma},\\
\ell(\alpha) + \ell(\beta) &= \ell(\gamma),
\end{aligned}
$$
and $\hat{\xi}_{ \alpha} \cdot \bar{\xi}_{ \beta}=0$ otherwise.

There is one subtle point here about comparing
the $\NN$-gradings on $R=\gr_I A$ and
the right $R$-module $\gr_I J$.  An equivalence class
$\beta=\widetilde{[x,y]}$ in $\Int(P)/\!\!\sim$ indexes the 
element $\bar{\xi}_{ \beta}$ which is homogeneous of degree
$\ell(\alpha)=\ell[x,y]$ for the $\NN$-grading on $R=\gr_I A$.
However, an equivalence class
$\alpha=\widetilde{[x,y]}$ lying in the monomial right ideal $J$
indexes an element $\hat{\xi}_{ \alpha}$ whose degree
for the $\NN$-grading on $R=\gr_I J$ is {\it not} given
by the length $\ell(\alpha)$.  Rather, if one recalls
from \eqref{J-restricted-interval} the definition of the subposet
$$
(x,y)|_J :=\{z \in (x,y): \widetilde{[x,z]} \in J\}
$$
inside $P$, then $\hat{\xi}_{ \alpha}$ has degree in $\gr_I J$ given
by the largest value $d$ in a chain
$$
x < z_0  < z_1 < \cdots < z_{d-1} < z_d =:y
$$
that has $z_i$ in $(x,y)|_J$ for $i=0,1,\ldots,d-1$.  
Rephrased, this says that the degree of $\hat{\xi}_{ \alpha}$ 
in $\gr_I J$ is {\it one more than} the dimension of 
the order complex $\Delta((x,y)|_J)$.

We now proceed to compute $\Tor^{R}(\gr_I J,A_0)$ in terms of the 
relative simplicial homology of certain subcomplexes of
these order complexes $\Delta\left( (x,y)|_J \right)$.  This
generalizes 
a result of Cibils \cite[Proposition 2.1]{Cib} for
incidence algebras of finite posets $P$, and a result of Laudal
and Sletsj{\o}e \cite{LS} for affine semigroup rings.

Note that, since $R$, $A_0$ and $\gr_I J$ are $\ZZ$-graded, there is also a 
$\ZZ$-grading on $\Tor_i^R(\gr_I J,A_0)$.  Denote its
$j^{th}$-homogenous component by $\Tor_i^R(\gr_I J,A_0)_j$.

\begin{theorem}
\label{thm:tor-computation}
Assume the pair $(P,\sim)$ satisfies Axioms~\ref{ax:order-compatibility}, 
\ref{ax:invariance}, \ref{ax:finiteness}, \ref{ax:concatenation}.
Let $R=\gr_I A$, where $A=k[P]_\red$.
Let $J$ be a monomial right ideal of $A$,
and consider $\gr_I J$ as a right $R$-module.

Then one has an $A_0-A_0$-bimodule decomposition
\begin{equation}
\label{thm-1st-assertion}
\Tor_i^R(\gr_I J,A_0)_j \cong \bigoplus_{(\tilde{x},\tilde{y}) \in P/\!\sim \times P/\!\sim} 
  \Tor_i^R(\xi_{\tilde{x}}\gr_I J,k_{\tilde{y}})_j
\end{equation}
with
\begin{equation}
\label{thm-2nd-assertion}
\Tor_i^R(\xi_{\tilde{x}}\gr_I J,k_{\tilde{y}})_j
= \bigoplus_{\substack{\alpha \in J:\\
                s(\alpha) =\tilde{x}\\ t(\alpha) = \tilde{y}}}  
\tilde{H}_{i-1}(\Delta_{\alpha}^{\langle j-1 \rangle} ,
                  \Delta_{\alpha}^{\langle j \rangle} ;k) .
\end{equation}
Here $\Delta_{\alpha}:=\Delta( (x(\alpha),y(\alpha))|_J )$,
where the interval $[x(\alpha),y(\alpha)]$ is the fixed representative from \eqref{eqn:xalpha-yalpha} 
of the equivalence  class $\alpha$.
\end{theorem}
Recall that $\Delta^{\langle j \rangle}$ denotes the $j^{th}$ sequential layer of $\Delta$;
see Definition~\ref{defn:sequential-Cohen-Macaulay}.
\begin{proof}
Write $R=A_0 \oplus \mm$, where $\mm:=\oplus_{d \geq 0} I^d/I^{d+1}$.
Start with the normalized bar resolution $(B,\partial)$ as a projective
resolution for $A_0$ as a left $R$-module \cite[Ch. X, \S2]{Mac}
$$
\begin{matrix}
\cdots  \rightarrow B_i \overset{\partial_i}{\rightarrow} B_{i-1} 
\overset{\partial_{i-1}}{\rightarrow} \cdots 
\overset{\partial_2}{\rightarrow} B_1 \overset{\partial_1}{\rightarrow} B_0
\end{matrix}
$$
having  
$$
\begin{aligned}
B_i&=R \otimes_{A_0} (\mm^{\otimes_{A_0} i}) \\
&=R \otimes_{A_0} \underbrace{\mm \otimes_{A_0} \cdots \otimes_{A_0} \mm}_{i \text{ times}}
\end{aligned}
$$
and differential
\begin{equation}
\label{bar-complex-differential}
\begin{aligned}
\partial_i (r \otimes m_1 \otimes \cdots \otimes m_i)
&:= r m_1\otimes m_2 \otimes \cdots \otimes m_i \\
& \qquad +\sum_{\ell=1}^{i-1}(-1)^\ell r \otimes m_1 \otimes \cdots 
                          \otimes (m_\ell m_{\ell+1}) \otimes \cdots \otimes m_i.
\end{aligned}
\end{equation}
Note that in \cite{Mac} the modules 
$B_i$ are called {\em relatively free} with respect to the 
pair of categories $R-Mod$ and $A_0-Mod$.  Because $A_0$ is not always a field, 
the $B_i$'s need not be $R$-free; but they are projective over $R$.

One applies $\gr_I J \otimes_R -$ to $(B,\partial)$, obtaining a complex
$(\gr_I J \otimes_R B, 1_{\gr_I J}\otimes \partial)$
whose homology computes $\Tor^R(\gr_I J,A_0)$.

For the assertion \eqref{thm-1st-assertion},
note that $\gr_I J$ carries the extra structure of a left $A_0$-module,
and $A_0$ carries the extra structure of a right $A_0$-module.
Such structures are easily seen to commute with the
above constructions, giving an $A_0-A_0$-bimodule structure
on $\Tor^R(\gr_I J,A_0)$.  This implies that the complex 
\begin{equation}
\label{tensored-bar-complex}
\begin{aligned}
(C,d)&:=(\xi_{\tilde{x}}\gr_I J \otimes_R B \otimes_{A_0} A_0 \xi_{\tilde{y}}, 
          \quad 1 \otimes \partial  ) \\
 &:=(\xi_{\tilde{x}}\gr_I J \otimes_R B \otimes_{A_0} k_{\tilde{y}}, 
          \quad 1 \otimes \partial  )
\end{aligned}
\end{equation}
whose homology computes $\Tor^R(\xi_{\tilde{x}} \gr_I J , A_0\xi_{\tilde{y}})$
is obtained from the one whose homology computes $\Tor^R(\gr_I J,A_0)$ 
by multiplying on the left and right by $\xi_{\tilde{x}}$, and $\xi_{\tilde{y}}$, 
respectively. From this \eqref{thm-1st-assertion} follows.

For the assertion \eqref{thm-2nd-assertion}, we examine
the complex $(C,d)$ in \eqref{tensored-bar-complex} more closely.
One finds that in homological degree $i$ it has the $k$-vector
space
\begin{equation}
\label{i-th-homological-degree}
\begin{aligned}
C_i
&=\xi_{\tilde{x}} \gr_I J \otimes_R R \otimes_{A_0} 
      (\mm^{\otimes_{A_0}i}) \otimes_{A_0} k_{\tilde{y}} \\
&\cong \xi_{\tilde{x}} \gr_I J \otimes_{A_0} 
      (\mm^{\otimes_{A_0}i}) \otimes_{A_0} k_{\tilde{y}}
\end{aligned}
\end{equation}
and that the differential $d_i$ takes the form
\begin{equation}
\label{tensored-differential}
\begin{aligned}
&d_i (\xi_{\tilde{x}} r \otimes m_1 \otimes \cdots \otimes m_i \otimes \xi_{\tilde{y}}) \\
&\quad = (\xi_{\tilde{x}} r \cdot m_1)\otimes m_2 \otimes \cdots \otimes m_i\xi_{\tilde{y}} \\
&\quad \qquad + \sum_{\ell=1}^{i-1}(-1)^\ell \xi_{\tilde{x}} r \otimes m_1 \otimes \cdots 
                          \otimes (m_\ell m_{\ell+1}) \otimes \cdots \otimes m_i \xi_{\tilde{y}}.
\end{aligned}
\end{equation}

Proposition~\ref{prop:tensor-bases} implies 
that $C_i$ from \eqref{i-th-homological-degree} decomposes as
an $A_0-A_0$-bimodule into $C_i = \oplus_{\alpha} C_i^\alpha$,
where $\alpha$ runs through the equivalence classes in $\Int(P)/\!\!\sim$
for which $s(\alpha)=\tilde{x}, t(\alpha)=\tilde{y}$, and
where $C^\alpha_i$ is described as follows.

If we let $x_0:=x(\alpha)$ and $y_0:=y(\alpha)$, then 
$C^\alpha_i$ has $k$-basis indexed by chains of the form
\begin{equation}
\label{typical-chain}
x_0 <  x_1 < \cdots < x_i < x_{i+1}=y_0,
\end{equation} 
where $\widetilde{[x_0,x_1]}$ lies in $J$.  

Since $J$ is a right ideal, this means that
each $\widetilde{[x_0,x_\ell]}$ lies in $J$ for $\ell=1,2,\ldots,i+1$, i.e.
the elements $x_1,\ldots,x_i$ lie in the subposet $(x_0,y_0)|_J$.  
Furthermore, when one restricts to the $j^{th}$-homogeneous component $C^\alpha_{i,j}$
for the $\NN$-grading on $C^\alpha_i$, one finds that
$C^\alpha_{i,j}$ has $k$-basis indexed by the chains 
as in \eqref{typical-chain} satisfying the extra condition that 
$$
\deg_{\gr_I J}\left( \hat{\xi}_{\widetilde{[x_0,x_1]}} \right) + 
        \deg_R \left( \bar{\xi}_{\widetilde{[x_1,x_2]}} \right) + 
 \cdots +\deg_R \left( \bar{\xi}_{\widetilde{[x_i,x_{i+i}]}} \right)  = j.
$$
Equivalently, since $\hat{\xi}_{\widetilde{[x_0,x_1]}}$ has
degree $1+\dim \Delta( (x_0,x_1)|_J )$ in $\gr_I J$, and
$\bar{\xi}_{\widetilde[x_m,x_{m+1}]}$ has degree $\ell[x_m,x_{m+1}]$ in $R$,
this says that
$$
\dim \Delta( (x_0,x_1)|_J ) +\ell[x_1,x_2] + \cdots +\ell[x_i,x_{i+1}] = j-1.
$$
This means that the chain $x_1 < \cdots < x_i$ inside the
subposet $(x_0,y_0)|_J$ extends to a chain with $j$ elements,
but cannot be extended to a chain with $j+1$ elements.  
Such a chain corresponds to an $(i-1)$-simplex of
$\Delta_{\alpha}$ which lies inside some $(j-1)$-simplex but
no $j$-simplex, or a face
of $\Delta_{\alpha}^{\langle j-1 \rangle}$ not lying in
$\Delta_{\alpha}^{\langle j \rangle}$, or a basis element in the
$(i-1)^{st}$ relative (reduced, simplicial) chain group for the pair 
$(\Delta_{\alpha}^{\langle j-1 \rangle},\Delta_{\alpha}^{\langle j \rangle})$.

One then checks that the differential \eqref{tensored-differential}
on $C^\alpha_{i,j}$ corresponds
to the relative simplicial boundary map for this pair, completing the proof.
\end{proof}

\section{Characterizing Koszulness}
\label{sec:koszulity}

Our goal is to prove Theorem~\ref{thm:main},
and then deduce both Corollary~\ref{cor:main-algebra}
(giving Theorems~\ref{thm:Polo-Woodcock}, \ref{thm:PRS}) 
and Theorem~\ref{thm:EHHRW}.

\subsection{Reformulating sequential acyclicity}

The following reformulation of sequential acyclicity 
is tailored to our purposes.  Although it may be new, it is very
closely related to Duval's remarks in \cite[Remark on p.6]{Duv}.

\begin{proposition}
\label{prop:new-seq-C-M-defn}
The following are equivalent for a simplicial complex $\Delta$:
\begin{enumerate}
\item[(a)]
$\Delta$ is sequentially acyclic over $k$.
\item[(b)]
For every $i,j$ with $-1 \leq i < j$ one has the following
vanishing of relative reduced homology 
\begin{equation}
\label{relative-layer-vanishing}
\tilde{H}_i(\Delta^{\langle j \rangle},\Delta^{\langle j+1 \rangle};k)=0.
\end{equation}
\end{enumerate}
\end{proposition}
\begin{remark}
\rm \ \\
It is equivalent in (b) to require the vanishing in \eqref{relative-layer-vanishing} 
to occur for every $i,j \geq -1$ with $i \neq j$:
the complexes $\Delta^{\langle j \rangle},\Delta^{\langle j+1 \rangle}$ have
the same $i$-dimensional  faces for $i>j$ so one always has
the vanishing in \eqref{relative-layer-vanishing} for
$i > j$.
\end{remark}

\begin{proof}
For both implications consider part of
the long exact sequence for the pair 
$(\Delta^{\langle j \rangle},\Delta^{\langle j+1 \rangle})$:
\begin{equation}
\label{long-exact-sequence}
\cdots \rightarrow
H_{i}(\Delta^{\langle j+1\rangle})
\rightarrow 
H_i(\Delta^{\langle j\rangle})
\rightarrow
H_i(\Delta^{\langle j\rangle},\Delta^{\langle j+1\rangle})
\rightarrow
H_{i-1}(\Delta^{\langle j+1\rangle})
\rightarrow \cdots
\end{equation}
\noindent
{\sf (a) implies (b)}:
Assume (a), and let $i,j$ satisfy $-1 \leq i < j$.
Then the second and fourth terms in \eqref{long-exact-sequence}
both vanish due to hypothesis (a), and hence
the third term also vanishes.

\vskip.1in
\noindent
{\sf (b) implies (a)}:
Assume (b), and fix $i \geq -1$.  We prove the vanishing
$H_i(\Delta^{\langle j\rangle};k)=0$ for each $j > i$ asserted
in (a) by descending induction on $j$.

In the base case, $j=\dim \Delta$, so that 
$\Delta^{\langle j+1\rangle} =\varnothing$
and hence
$$
H_i(\Delta^{\langle j\rangle};k)
= H_i(\Delta^{\langle j\rangle},\Delta^{\langle j+1\rangle};k)
=0
$$
by hypothesis (b).  In the inductive step, note
that in \eqref{long-exact-sequence}, the third term
vanishes due to hypothesis (b), while the first one vanishes by the descending induction on $j$.
Hence the second term vanishes.
\end{proof}

\subsection{Proof of Theorem~\ref{thm:main}}

Recall the statement of the theorem.

\vskip.1in
\noindent
{\bf Theorem ~\ref{thm:main}.}
{\it 
Let $P$ be a poset and $\sim $ an equivalence relation on $\Int(P)$, satisfying Axioms
\ref{ax:order-compatibility}, \ref{ax:invariance}, \ref{ax:finiteness}, \ref{ax:concatenation}.  
Let $J$ be a monomial right ideal in $k[P]_\red$.

Then $J$ is a nongraded Koszul right $k[P]_\red$-module
if and only if for every equivalence class $\widetilde{[x,y]}$ in $J$
the poset $(x,y)|_J$ is sequentially acyclic over $k$.
}

\vskip.1in

In its proof, we use a standard reformulation of Koszulness
in terms of $\Tor$;  see e.g. Woodcock \cite[Theorem 2.3]{Woc},
or \cite[Prop. 2.1.3]{BGS}.

\begin{proposition}
\label{prop:Koszulness-via-tor}
For an $\NN$-graded ring $R$ with $R_0$ semisimple,
and $M$ a $\ZZ$-graded right $R$-module bounded from below (i.e. $M_i=0$ for $i\ll 0$), one
has that $M$ is a Koszul $R$-module if and only if
$\Tor^R_i(M,R_0)_j=0$ for every $i,j \in \ZZ$ with $i \neq j$. 
\end{proposition}

\begin{proof}(of Theorem~\ref{thm:main})
\newline
Assume that the pair $(P,\sim)$ satisfies Axioms~\ref{ax:order-compatibility},
~\ref{ax:invariance},~\ref{ax:finiteness},~\ref{ax:concatenation}.
Let $R=\gr_I A$ where $A=k[P]_\red$, and let $J$ be a monomial right ideal of $A$.

By definition, the right $A$-module $J$ is nongraded Koszul if and only if
the graded right $R$-module $\gr_I J$ is a Koszul $R$-module.
By Proposition~\ref{prop:Koszulness-via-tor} and
Theorem~\ref{thm:tor-computation}, 
this occurs if and only if for every equivalence class
$\alpha=\widetilde{[x,y]}$ in $J$,
the order complex $\Delta_\alpha=\Delta( (x,y)|_J)$
satisfies the vanishing condition
$$
\tilde{H}_{i-1}(\Delta_\alpha^{\langle j-1 \rangle},
                  \Delta_\alpha^{\langle j \rangle};k) = 0 \text{ for }i \neq j \geq 0.
$$
Comparison with Proposition~\ref{prop:new-seq-C-M-defn} shows 
that this is equivalent to every  
$\widetilde{[x,y]}$ in $J$
having $(x,y)|_J$ sequentially acyclic over $k$.  
\end{proof}

\subsection{Deduction of Corollary~\ref{cor:main-algebra}}
\label{sec:deducing-main-corollary}

We recall here the statement of the corollary.

\vskip .1in
\noindent
{\bf Corollary~\ref{cor:main-algebra}.}
{\it
Let $P$ be a poset and $\sim $ an equivalence relation on $\Int(P)$, satisfying Axioms
\ref{ax:order-compatibility}, \ref{ax:invariance}, \ref{ax:finiteness}, \ref{ax:concatenation}.  

Then $k[P]_\red$ is a nongraded Koszul ring 
if and only if every open interval $(x,y)$ in $P$
is sequentially Cohen-Macaulay over $k$.
}

\vskip.1in
Corollary~\ref{cor:main-algebra} will be deduced as the special case $J=I$ of
Theorem~\ref{thm:main}.  There are two small observations that allow
this.

The first is a reformulation of the sequential Cohen-Macaulay property
for posets, due to Bj\"orner, Wachs and Welker \cite{BWW}.
Given a poset $P$, let $\hat P := P \sqcup \{ \hat 0, \hat 1\}$
be the bounded poset obtained by adding a new minimum element $\hat 0$ and maximum
element $\hat 1$.

\begin{proposition} (\cite[Corollary 3.5]{BWW})
\label{prop:BWW-poset-criterion}
A poset $P$ has its order complex $\Delta P$ sequential Cohen-Macaulay over
$k$ if and only if every open interval $(x,y)$ in $\hat P$
is sequentially acyclic over $k$.
\end{proposition}

The second observation is a standard 
argument.

\begin{proposition}
\label{prop:dimension-shifting}
Let $A$ be a ring with a decomposition
$A=A_0 \oplus I$ in which $A_0$ is a semisimple ring and
$I$ a two-sided ideal.

Then $A$ is a nongraded Koszul ring if and only $I$ is a nongraded
Koszul right $A$-module.
\end{proposition}
\begin{proof}
Note that when one decomposes the associated graded ring $\gr_I A$ as
$$
\gr_I A = \underbrace{A/I}_{A_0:=} \oplus 
            \underbrace{I/I^2 \oplus I^2/I^3 \oplus \cdots}_{\mm:=},
$$
the irrelevant ideal $\mm$ regarded as a right $\gr_I A$-module
is isomorphic to the right $\gr_I A$-module $(\gr_I I)(-1)$,
that is, its degree zero homogeneous component $I/I^2$ lives in degree one
of $\mm$.
This leads to the following equivalences, justified below:
$$
\begin{aligned}
&I\text{ is a nongraded Koszul right }A\text{-module } \\
&\overset{(1)}{\Leftrightarrow} M:=\gr_I I\text{ is a Koszul right }\gr_I A\text{-module } \\
&\overset{(2)}{\Leftrightarrow} \Tor_i^{\gr_I A}(\gr_I I,A_0)_j = 0 \text{ for }j \neq i\\
&\overset{(3)}{\Leftrightarrow} \Tor_i^{\gr_I A}(\mm,A_0)_j = 0 \text{ for }j \neq i+1\\
&\overset{(4)}{\Leftrightarrow} \Tor_i^{\gr_I A}(A_0,A_0)_j = 0 \text{ for }j \neq i\\
&\overset{(5)}{\Leftrightarrow} A_0:=A/I\text{ is a Koszul right }\gr_I A\text{-module } \\
&\overset{(6)}{\Leftrightarrow} A\text{ is a nongraded Koszul ring}.\\
\end{aligned}
$$
Equivalences (1), (6) are by definition, while equivalences (2), (5) use
Proposition~\ref{prop:Koszulness-via-tor}.  Equivalence (3) uses the
observation that $\mm=(\gr_I I)(-1)$ from above.  

This leaves only equivalence (4).  For $i\geq 0$, this 
follows from the usual suspension isomorphism
\begin{equation}
\label{suspension-isomorphism}
\Tor_i^{R}(\mm,R_0) \cong \Tor^R_{i+1}(R_0,R_0)
\end{equation}
for graded rings $R$, that comes from the long exact sequence
in $\Tor^R(-,R_0)$ associated to the short exact sequence 
$$
0 \rightarrow \mm \rightarrow R \rightarrow R_0 \rightarrow 0.
$$
To deduce the $i=-1$ case of equivalence (4), note that 
$\Tor_{-1}^{\gr_I A}(\mm,A_0)  = 0$ while 
$$
\Tor_0^{\gr_I A}(A_0,A_0)_j=(A_0 \otimes_{\gr_I A} A_0)_j=(A_0)_j=0 \text{  for }j \neq 0.
$$
\end{proof}

\begin{proof}(of Corollary~\ref{cor:main-algebra})

Proposition~\ref{prop:dimension-shifting} says that
$k[P]_\red$ is a nongraded Koszul ring if and only if
$I$ is a nongraded Koszul right $k[P]_\red$-module.
Taking $J=I$ in Theorem~\ref{thm:main}, this is equivalent
to the poset $P$ having every open interval $(x,y)$ sequentially
acyclic over $k$.   Comparing with
Proposition~\ref{prop:BWW-poset-criterion}, and noting that
every open interval in the bounded poset 
$[x,y]=\widehat{(x,y)}$ is an open interval in $P$ itself,
one concludes that this is equivalent to every open interval
$(x,y)$ in $P$ being sequentially Cohen-Macaulay over $k$.
\end{proof}

\begin{remark}
\rm \ \\
As observed in the Introduction, Theorems~\ref{thm:Polo-Woodcock}
and \ref{thm:PRS} are simply the special cases of Corollary~\ref{cor:main-algebra}
in which $P$ is either a graded finite poset with $\sim$ the trivial equivalence
relation on its intervals, or 
$P$ is the poset structure on a graded pointed affine semigroup $\Lambda$ and
$\sim$ is the equivalence relation
\eqref{semigroup-equivalence} on its intervals.
\end{remark}

\begin{remark}
\rm \ \\
The case $J=I$ in Theorem~\ref{thm:tor-computation} along with
the suspension isomorphism \eqref{suspension-isomorphism} immediately implies
the following generalization of Theorems~\ref{thm:Polo-Woodcock} and
\ref{thm:PRS}.
\end{remark}

\begin{corollary}
\label{cor:ring-tor}
Let $P$ be a poset and $\sim $ an equivalence relation on $\Int(P)$, 
satisfying Axioms \ref{ax:order-compatibility}, \ref{ax:invariance}, 
\ref{ax:finiteness}, \ref{ax:concatenation}. 
Let $R$ be the associated graded ring $R=\gr_I A$ for 
$A=k[P]_\red$.

Then 
$\Tor_0^R(k_{\tilde{x}},k_{\tilde{y}}) = k$ is concentrated in degree $0$ if
$\tilde{x}=\tilde{y}$, and vanishes otherwise, 
while for $i \geq 1$ one has
$$
\Tor_i^R(k_{\tilde{x}},k_{\tilde{y}})_j 
\cong
\bigoplus_{\substack{\alpha \in \Int(P)/\!\sim \\ s(\alpha)=\tilde{x} \\ t(\alpha)=\tilde{y}}} 
\tilde{H}_{i-2}( \Delta_\alpha^{\langle j-2 \rangle},
\Delta_\alpha^{\langle j-1 \rangle} ; k ).
$$
Here $\Delta_\alpha$ is the order complex of the open interval $(x(\alpha),y(\alpha))$,
where $[x(\alpha),y(\alpha)]$ is the representative  of $\alpha$ fixed in \eqref{eqn:xalpha-yalpha}.
\end{corollary}

\subsection{Deduction of Theorem~\ref{thm:EHHRW}}
\label{sec:deducing-EHHRW}

We recall here the statement of the theorem.

\vskip.1in
\noindent
{\bf Theorem~\ref{thm:EHHRW}}
{\it
The Stanley-Reisner ideal $I_\Delta$ in $A=k[x_1,\ldots,x_n]$
has componentwise linear resolution as an $A$-module if and only if 
the Alexander dual complex $\Delta^\vee$ is sequentially Cohen-Macaulay over $k$.
}

\vskip.1in

As explained in Section~\ref{sec:unification} of the Introduction, to deduce
Theorem~\ref{thm:EHHRW} from Theorem~\ref{thm:main},
one should consider $\Lambda=\NN^n$ as an affine semigroup, 
with $P=\Lambda$ and $\sim$ the equivalence relation \eqref{semigroup-equivalence}.
Here 
$$
A=k[P]_\red=k[x_1,\ldots,x_n]= \gr_I A.
$$

Fix a Stanley-Reisner ideal $J=I_\Delta$ in $A$, associated to the simplicial
complex $\Delta$, and Alexander dual complex $\Delta^\vee$.
As explained in Section~\ref{sec:unification}, it only remains to prove
the following.

\begin{proposition}
In the above setting, 
all links of faces in $\Delta^\vee$ are sequentially acyclic over
$k$ if and only if
all subposets $(x,y)|_J$ for $\widetilde{[x,y]}$ in $J$
are sequentially acyclic over $k$.
\end{proposition}

\begin{proof}
First note that any interval $[x,y]$ in $P=\NN^n$
is equivalent to an interval $[1,x^\alpha]$ 
(under the divisibility ordering) for a unique monomial 
$
x^\alpha= x_1^{\alpha_1} \cdots x_n^{\alpha_n}.
$
Furthermore, $\widetilde{[x,y]}=\widetilde{[1,x^\alpha]}$ 
lies in $J$ if and only if $x^\alpha$ lies in the monomial ideal $I_\Delta$.
Denote by  
$$
\sqrt{x^\alpha} : = \prod_{i: \alpha_i > 0} x_i
$$
the unique squarefree monomial having the same variable support
as $x^\alpha$.

We analyze sequential layers of the poset $(1,x^\alpha)|_J$ 
based on two cases for $\sqrt{x^\alpha}$.

\vskip.1in
\noindent
{\sf Case 1.} $x^\alpha=\sqrt{x^\alpha}$, that is, $x^\alpha$ is squarefree.

In this case, since $x^\alpha$ lies in $J=I_\Delta$, one must have
$x^\alpha=\prod_{i \in [n] \setminus F} x_i$ for a unique 
face of $\Delta^\vee$.  
Furthermore, every monomial $x^{\alpha'}$ 
lying in $(1,x^\alpha)|_J$ is also squarefree, and hence
of the form $x^{\alpha'}=\prod_{i \in [n] \setminus F'} x_i$
for a unique face $F'$ of $\Delta^\vee$ that contains $F$.
Therefore the map sending $x^{\alpha'}$ to $F'\setminus F$
is a poset isomorphism from $(1,x^\alpha)|_J$ to the poset
of nonempty faces of $\link_{\Delta^\vee}(F)$.  Consequently,
the order complex $\Delta( (1,x^\alpha)|_J )$ is simplicially
isomorphic to the barycentric subdivision of $\link_{\Delta^\vee}(F)$.
For every $j$, this also induces a simplicial isomorphism
between their $j^{th}$ sequential layers, so that sequential
acyclicity for the two are equivalent.

\vskip.1in
\noindent
{\sf Case 2.} $x^\alpha \neq \sqrt{x^\alpha}$.

In this case, we wish to show that the $j^{th}$ sequential layer
$\Delta'$ of the order complex $\Delta((1,x^\alpha)|_J)$
is contractible.  Essentially, we will show that it is
star-shaped with respect to the vertex $v:=\sqrt{x^\alpha}$.  

A face of $\Delta'$ is a chain $c$ in $(1,x^\alpha)|_J$
that can be extended to a chain with at least $j+1$ elements.
The interval $[1,x^\alpha]$ under divisibility is a 
{\it distributive lattice}, and hence this chain $c$ together
with the element $v=\sqrt{x^\alpha}$ generates $L(c,v)$ a distributive 
sublattice of $[1,x^\alpha]$.  Let $\Delta_c$ denote the
order complex of the subposet $L(c,v) \cap (1,x^\alpha)|_J$.
We claim that the complex $\Delta_c$, which obviously contains
both the simplex $c$ and the vertex $v$, 
\begin{enumerate}
\item[$\bullet$] triangulates a ball (of dimension at least $j$), and
\item[$\bullet$] is a subcomplex of $\Delta'$.
\end{enumerate}
These two claims would imply that $\Delta'$ is contractible:
the family of balls $\Delta_c$ would provide contractible
carriers for the identity map $1_{\Delta'}$ and the constant
map $f: \Delta' \rightarrow \{v\}$, showing that they are homotopic,
by the Contractible Carrier Lemma (see e.g. 
\cite[Lemma 10.1]{Bj-top},
\cite[Chapter 3, \S2.4]{BGMW}).

To see the two claims, start by noting that the bottom element $x^{\beta}$ on the chain
$c$ must lie in $J$ and must divide $x^\alpha$.  Therefore 
$\sqrt{x^\beta}$ also lies in $J$, and divides $v=\sqrt{x^\alpha}$,
so it provides a minimum element of $L(c,v)$.  Since this
minimum element of $L(c,v)$ lies in $(1,x^\alpha)|_J$,
the rest of the distributive lattice $L(c,v)$ lies in $(1,x^\alpha)|_J$.
Therefore the poset $L(c,v) \cap (1,x^\alpha)|_J$ is either 
\begin{enumerate}
\item[$\bullet$] 
all of $L(c,v)$, if $x^\alpha$ is not in $L(c,v)$, or
\item[$\bullet$]
$L(c,v)$ with its top element removed, if $x^\alpha$ lies in 
$L(c,v)$.  
\end{enumerate}
In either case, the order complex $\Delta_c$ triangulates a ball:
a distributive lattice is pure shellable \cite{Bj} and has each
codimension one face lying in at most two facets, so that its proper
part triangulates a ball or sphere \cite[Proposition 4.7.22]{BLSWZ},
and when one puts back its bottom and/or top element, it triangulates a ball.
This ball contains the chain $c$ and hence has dimension at least $j$.
Since this ball is pure of dimension at least $j$, and it is also a subcomplex of $(1,x^\alpha)|_J$,
it is a subcomplex of the $j^{th}$ sequential layer $\Delta'$.
\end{proof}

\section{Some Koszul numerology}
\label{sec:numerology}

We discuss in this context how some numerology of Koszul rings
interacts with some numerology of sequentially
Cohen-Macaulay complexes.

The sequential Cohen-Macaulay numerology is part
of the $f$- and $h$-triangles of a simplicial complex introduced by Bj\"orner and 
Wachs \cite[Definition 3.1]{BW}, which we recall here.

\begin{definition}
\rm \ \\
For a finite simplicial complex $\Delta$, let $f_{i,j}$
denote the number of simplices $F$ in $\Delta$
having $j$ vertices, and for which the largest-dimensional
face containing $F$ has $i$ vertices.
Define
$$
h_{i,j}:=\sum_{k=0}^j (-1)^{j-k} \binom{i-k}{j-k} f_{i,k}.
$$
\end{definition}

It turns out that when $\Delta$ is sequentially Cohen-Macaulay
over any field, each entry $h_{i,j}$ is nonnegative,
having various algebraic/homological interpretation; see e.g.,
\cite[Corollary 6.2]{Duv}.  
We will be mainly interested here in a simple interpretation for its diagonal
entries $h_{i,i}$ that holds under the weaker assumption of sequential
acyclicity.

\begin{proposition}
\label{prop:h-triangle-diagonal}
When $\Delta$ is sequentially acyclic over $k$,
one has 
$$
h_{i,i}(\Delta) 
= \dim_k \tilde{H}_{i-1}(
            \Delta^{\langle i-1 \rangle}, \Delta^{\langle i \rangle}
          ;k).
$$
\end{proposition}
\begin{proof}
Note that $f_{i,j}$ counts the number of $(j-1)$-dimensional
faces of $\Delta^{\langle i-1 \rangle}$ that do not
lie in $\Delta^{\langle i \rangle}$.  Hence
$$
f_{i,j}
=\dim_k \tilde{C}_{j-1}(\Delta^{\langle i-1 \rangle}, 
                            \Delta^{\langle i \rangle};k ).
$$
and therefore
$$
\begin{aligned}
(-1)^{i-1} h_{i,i}
&= \sum_{j \geq 0} (-1)^{j-1} f_{i,j} \\
&= \sum_{j \geq 0} (-1)^{j-1} 
\dim_k \tilde{C}_{j-1}(\Delta^{\langle i-1 \rangle}, 
                         \Delta^{\langle i \rangle};k) \\
&=\sum_{j \geq 0} (-1)^{j-1} 
\dim_k \tilde{H}_{j-1}(\Delta^{\langle i-1 \rangle}, 
                           \Delta^{\langle i \rangle};k) \\
&= (-1)^{i-1} \dim_k 
\tilde{H}_{i-1}(\Delta^{\langle i-1 \rangle}, 
                  \Delta^{\langle i \rangle};k),
\end{aligned}
$$
where the last equality uses Proposition~\ref{prop:new-seq-C-M-defn}.
\end{proof}

The Koszul numerology requires an extra finiteness condition on the poset $P$ and
equivalence relation $\sim$, stronger than Axiom~\ref{ax:finiteness}:

\vskip.1in
\noindent
{\bf Axiom~\ref{ax:finiteness}${}^{+}$.}
{\it
For every integer $\ell \geq 0$, there are only
finitely many equivalence classes $\alpha$ in $\Int(P)/\!\!\sim$
with $\ell(\alpha)=\ell$.
}

\vskip.1in
\noindent
When this axiom holds, set $R:=\gr_I A$ for $A=k[P]_\red$,
and define a matrix 
$P(t)$ in $\ZZ[[t]]^{P/\!\sim \,\, \times \,\,  P/\!\sim}$
whose $(\tilde{x},\tilde{y})$-entry is
$$
P_{\tilde{x},\tilde{y}}(t)
:= \sum_i \dim_k \left(\xi_{\tilde{x}} R \xi_{\tilde{y}} \right)_i t^i
= \sum_{\substack{\alpha \in \Int(P)/\!\sim:\\ 
        s(\alpha) =\tilde{x}, \\t(\alpha) =\tilde{y}}} t^{\ell(\alpha)}.
$$

We will also need to assume the {\it $\Tor$-finiteness condition}
\begin{equation}
\label{Tor-finiteness-condition}
\dim_k \Tor_i^{R}(R_0,R_0) < \infty \text{ for all }i.
\end{equation}
Beilinson, Ginzburg and Soergel explain \cite[\S 2.11]{BGS} how condition
\eqref{Tor-finiteness-condition}
follows whenever $R$ is left-Noetherian as a $k$-algebra,
which one can easily check holds at least in our main 
examples:  $(P,\sim)$ with $P$ finite and $\sim$ trivial,
or $P=\Lambda$ a pointed affine semigroup and $\sim$ as in 
\eqref{semigroup-equivalence}.  
When condition \eqref{Tor-finiteness-condition} holds, define another matrix
$Q(t)$ in $\ZZ[[t]]^{ P/\! \sim \,\, \times \,\, P/\! \sim}$
whose $(\tilde{x},\tilde{y})$-entry is
$$
Q_{\tilde{x},\tilde{y}}(t) 
:= \sum_i \dim_k \Tor^R_i(k_{\tilde{x}},k_{\tilde{y}}) \,\, t^i .\\
$$
Using Corollary~\ref{cor:ring-tor}, one can rephrase this as
saying that $Q_{\tilde{x},\tilde{x}}(t):=1$, and for $\tilde{x} \neq \tilde{y}$, 
one has
$$
Q_{\tilde{x},\tilde{y}}(t) 
:=\sum_{\substack{\alpha \in \Int(P)/\! \sim:\\ 
        s(\alpha) = \tilde{x}, \\t(\alpha) = \tilde{y}}} \qquad
  \sum_{i,j \geq 1} \dim_k \tilde{H}_{i-2}(\Delta_\alpha^{\langle j-2 \rangle},
                                 \Delta_\alpha^{\langle j-1 \rangle};k) \,\, t^i.
$$

\begin{proposition}
Let $(P,\sim)$ be a poset and equivalence relation
satisfying Axioms~\ref{ax:order-compatibility},
\ref{ax:invariance}, \ref{ax:finiteness}${}^+$, \ref{ax:concatenation},
as well as condition \eqref{Tor-finiteness-condition}.

If $A=k[P]_\red$ is a nongraded Koszul ring then
one can re-express
\begin{equation}
\label{Q-re-expressed}
Q_{\tilde{x},\tilde{y}}(t) 
= \sum_{\substack{\alpha \in \Int(P)/\! \sim:\\ 
        s(\alpha) = \tilde{x}, \\t(\alpha) =\tilde{y}}} \qquad
  \sum_{i \geq 1} h_{i-1,i-1}(\Delta_\alpha) \,\, t^i
\end{equation}
and $P(t), Q(t)$ determine each other uniquely by the matrix equation 
\begin{equation}
\label{Koszul-matrix-equation}
P(t)Q(-t)=I.
\end{equation}
\end{proposition}

\begin{proof}
When $A$ is nongraded Koszul, Corollary~\ref{cor:ring-tor},
Proposition~\ref{prop:Koszulness-via-tor}, Proposition~\ref{prop:h-triangle-diagonal}
together imply the formula \eqref{Q-re-expressed} for $Q_{\tilde{x},\tilde{y}}(t)$.

When $R=\gr_I A$ is an $\NN$-graded Koszul ring,
the matrix equation \eqref{Koszul-matrix-equation} is a consequence
of \cite[Theorem 2.9]{Woc} or \cite[Theorem 2.11.1]{BGS}.
\end{proof}

\begin{example} 
\label{ex:numerology}
\rm \ \\
Let $P$ be the poset on $[8]=\{1,2,3,4,5,6,7,8\}$ from Figure~\ref{pic:exposet}.
  
\begin{figure} [htb] 
	\centering
	\includegraphics{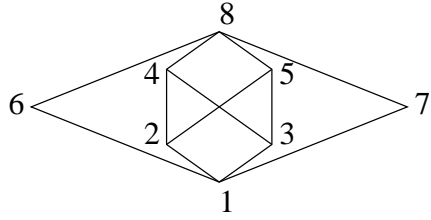}
	\caption{The poset $P$.}
	\label{pic:exposet}
\end{figure}

Considering the trivial equivalence relation $\sim$ on $\Int(P)$,
so that $k[P]_\red=k[P]$, the matrices $P(t), Q(t)$ are as follows:
$$
P(t)=\left[
\begin{matrix}
1&t&t&t^2&t^2&t&t&t^3\\
0&1&0&t&t&0&0&t^2\\
0&0&1&t&t&0&0&t^2\\
0&0&0&1&0&0&0&t\\
0&0&0&0&1&0&0&t\\
0&0&0&0&0&1&0&t\\
0&0&0&0&0&0&1&t\\
0&0&0&0&0&0&0&1
\end{matrix}
\right],
$$
$$
Q(t)=\left[
\begin{matrix}
1&t&t&t^2&t^2&t&t&2t^2+t^3\\
0&1&0&t&t&0&0&t^2\\
0&0&1&t&t&0&0&t^2\\
0&0&0&1&0&0&0&t\\
0&0&0&0&1&0&0&t\\
0&0&0&0&0&1&0&t\\
0&0&0&0&0&0&1&t\\
0&0&0&0&0&0&0&1
\end{matrix}
\right].
$$
For example, to compute the entry $Q_{1,8}(t)$, one
considers the open interval $(1,8)$.  The
order complex $\Delta_\alpha = \Delta(1,8)$ 
is the disjoint union of the two isolated vertices labelled $6,7$
with a $1$-sphere on the vertices labelled $2,4,3,5$ in cyclic order.
This means that 
\begin{enumerate}
\item[$\bullet$]
$\Delta_\alpha^{\langle -1 \rangle}, \Delta_\alpha^{\langle 0 \rangle}$ 
are both equal to all of $\Delta_\alpha$, while
\item[$\bullet$]
$\Delta_\alpha^{\langle 1 \rangle}$ is the $1$-sphere on vertices $2,4,3,5$, and
\item[$\bullet$]
$\Delta_\alpha^{\langle i \rangle}=\varnothing$ is the
empty complex for $i \geq 2$.  
\end{enumerate}
From this one can compute
$$
\tilde{H}_{i}(\Delta_\alpha^{\langle i \rangle}, 
\Delta_\alpha^{\langle i+1 \rangle}; k)  \cong 
\begin{cases}
0   &\text{ for }i=-1\\
k^2 &\text{ for }i=0\\
k   &\text{ for }i=1\\
0   &\text{ for }i \geq 2
\end{cases}
$$
and hence 
$$
\begin{aligned}
Q_{1,8}(t) 
&= 0 \cdot t^{-1+2} + 2 \cdot t^{0+2} 
   + 1 \cdot t^{1+2} + 0 \cdot t^{2+2} + 0 \cdot t^{3+2} + \cdots \\
&= 2t^2+t^3.
\end{aligned}
$$
One can then check that the matrix equation $P(t)Q(-t)=I_{8 \times 8}$ holds here.
\end{example}

\section{Characterizing Quadraticity}
\label{sec:quadraticity}

We wish to characterize combinatorially when 
$\gr_I A$ for $A=k[P]_\red$ is a quadratic ring,
generalizing a result of Woodcock \cite[Lemma 4.5]{Woc} 
for incidence algebras of graded posets $P$.
Woodcock's result uses the well-known notion of {\it gallery-connectedness}
for pure simplicial complexes.  We begin by developing this notion
and its properties in the nonpure setting.

\subsection{Sequential galleries and their properties}

\begin{definition}
\rm \ \\
In a simplicial complex $\Delta$, say that two faces $F, G$
(say with $\dim F \leq \dim G$) have a {\it sequential gallery}
connecting them, and write $F \rightarrow G$, if there exists a sequence
of faces of $\Delta$
\begin{equation}
\label{typical-gallery}
F = F_0,F_1,\ldots,F_t=G
\end{equation}
such that for each $i=1,\ldots,t$, either
\begin{enumerate}
\item[$\bullet$] $\dim F_{i-1} = \dim F_i= 1 + \dim (F_{i-1} \cap F_i)$, or
\item[$\bullet$] $F_{i-1} \subset F_i$ and $\dim F_i = \dim {F_{i-1}}+1$.
\end{enumerate}
When $\dim F = \dim G$, note that this forces all of the steps in
the sequential gallery to be of the first kind, in which case
the sequence is called a {\it gallery} connecting the faces $F, G$.
\end{definition}

\begin{lemma}
\label{lemma:sequential-gallery}
If $F \rightarrow G$ in a simplicial complex $\Delta$, then for
any subfaces $F' \subseteq F$ and $G' \subseteq G$, there is a sequential
gallery connecting them.  That is, either
$$
\begin{aligned}
F'\rightarrow G' & \text{ if }\dim F' \leq \dim G' \text{, or}\\
G'\rightarrow F' & \text{ if }\dim G' \leq \dim F'.
\end{aligned}
$$
\end{lemma}
\begin{proof}
Induct on the length $t$ of a sequential gallery as in \eqref{typical-gallery} 
connecting $F$ to $G$.  

In the base case $t=0$, one has $F=G$, so that $F',G'$ are subfaces of a
single ambient simplex.
By relabelling, one may assume without loss of generality  
that $\dim F' \leq \dim G'$.  Pick a subface $G'' \subseteq G'$
with the same dimension as $F'$.  It is easy to produce
a gallery $F' \rightarrow G''$ using downward induction on the
size of the symmetric difference $(F' \setminus G'') \cup (G'' \setminus F')$:
replace one-by-one the elements of $F' \setminus G''$
with the elements of $G'' \setminus F'$.
One can then follow this with a sequential gallery from $G'' \rightarrow G'$.

In the inductive step, there are two cases.  
\vskip .1in
\noindent
{\sf Case 1.} $\dim F_{t-1}=\dim F_t (=\dim G)$.

Replace $G$ with $F_{t-1}$, and replace the subface $G'\subseteq G$ with
a subface $F_{t-1}'\subseteq F_{t-1}$ having the same dimension as $G'$.
Now apply induction to produce a sequential gallery $F'\rightarrow F'_{t-1}$,
followed by a gallery $F'_{t-1} \rightarrow G'$.

\vskip .1in
\noindent
{\sf Case 2.} $\dim F_{t-1} < \dim F_t (=\dim G)$.

One has 
$$
\dim F' \leq \dim F \leq \dim F_{t-1}
$$
since $F'\subseteq F$ and since there is a gallery $F \rightarrow F_{t-1}$.
Hence there exists a subface $F'_{t-1} \subseteq F_{t-1}$ having the same
cardinality as $F'$.  Apply induction to get a gallery connecting
$F', F'_{t-1}$ in either direction, and also to get either a sequential gallery
$G'\rightarrow F'_{t-1}$ or $F'_{t-1} \rightarrow G'$.  Either of
these sequential galleries can be concatenated with the gallery
connecting $F',F'_{t-1}$ to give a sequential gallery connecting $F',G'$.
\end{proof}

\begin{proposition}
\label{prop:seq-gallery-equivalences}
The following are equivalent for a simplicial complex $\Delta$,
and define the notion for $\Delta$ to be 
{\it sequential-gallery-connected (SGC)} :
\begin{enumerate}
\item[(i)]
Every pair of faces has a sequential gallery connecting them.
\item[(ii)]
Every pair of facets has a sequential gallery connecting them.
\item[(iii)]
Every pair of faces of the same dimension has a gallery connecting them.
\end{enumerate}
\end{proposition}
\begin{proof}
\noindent
{\sf (i) implies (ii)}: Trivial.
\newline
\noindent
{\sf (ii) implies (iii)}:
This follows from Lemma~\ref{lemma:sequential-gallery}
since every face lies in a facet.
\newline
\noindent
{\sf (iii) implies (i)}:
Assume (iii) and start with any pair of faces $F, G$.
Without loss of generality assume $\dim F \leq \dim G$.
Then there exist a subface $G' \subseteq G$ having the same dimension
as $F$, and a gallery $F \rightarrow G'$ by (iii).  
Following this by a sequential gallery $G'\rightarrow G$
gives a sequential gallery $F \rightarrow G$.
\end{proof}

When one insists that the SGC property is hereditary in the
sense of being inherited by all {\it links} of faces, it
becomes equivalent to a sequential version of connectivity.

\begin{proposition}
\label{prop:connectivity-equivalences}
The following are equivalent for a simplicial complex $\Delta$:
For every face $F$ of $\Delta$, the link $K:=\link_\Delta(F)$
\begin{enumerate}
\item[(i)]
is SGC.
\item[(ii)]
has $K^{\langle j \rangle}$ connected for all $j$ in the range $1 \leq j \leq \dim K$.
\item[(iii)]
has $\tilde{H}_0(K^{\langle j \rangle},K^{\langle j+1\rangle};k)=0$ 
for some field $k$ and 
for all $j$ in the range $1 \leq j \leq \dim K$.
\end{enumerate}
\end{proposition}
\begin{proof}

\noindent
{\sf (ii) if and only if (iii):}
Since vanishing of zero-dimensional reduced
homology (over any field $k$) is equivalent to connectedness,
this equivalence follows
from an easy long-exact sequence argument,
as in the proof of Proposition~\ref{prop:new-seq-C-M-defn}.

\noindent
{\sf (i) implies (ii):}
Assume every link of $\Delta$ is SGC.  Given two vertices $v, v'$ lying 
in $K^{\langle j \rangle}$ for some link $K$ and some $j \geq 1$,
pick $j$-simplices $F, F'$ containing $v,v'$ respectively.
Since $K$ is SGC there exists a gallery $F \rightarrow F'$ in $K$.
Because $\dim F = \dim F'=j$, this gallery lies in $K^{\langle j \rangle}$,
and hence also gives a path connecting $v,v'$ inside $K^{\langle j \rangle}$.

\noindent
{\sf (ii) implies (i):}
Assume every link $K$ of $\Delta$ has $K^{\langle j \rangle}$ connected
for all $j \geq 1$, and we want to show that every link $K$ is SGC.
Proceed by induction on $\dim K$, with the base case $\dim K = 0$ being trivial.

In the inductive step, let $F,G$ be faces in $K$ 
of the same dimension $d$, and we want to
find a gallery between them.  

If $d=0$ this is trivial, so assume that $d > 0$.  
Since both $F, G$ lie in $K^{\langle d \rangle}$, which is connected
by (ii), one can pick vertices $f,g$ lying in $F,G$,
and there will exist a path $f=f_0,f_1,\dots,f_t=g$ 
along edges in $K^{\langle d \rangle}$.

Proceed by a second (inner) induction on the length $t$ of this
path.  In the base case $t=0$, both $F,G$ contain $f=f_0=g$.
Since $F \setminus \{f_0\}, G \setminus \{f_0\}$ are
faces of the link $K':=\link_K(\{f_0\})=\link_\Delta(F\cup\{f_0\})$ and $\dim K' < \dim K$, by the outer induction on $d$,
this $K'$ is SGC and there exists a gallery 
$F \setminus \{f_0\} \rightarrow G \setminus \{f_0\}$ in $K'$.
This leads to a gallery $F \rightarrow G$ in $K$.

In the second (inner) inductive step where $t \geq 1$,
since the $\{f_0,f_1\}$ lies in $K^{\langle d \rangle}$, it lies
in some $d$-face $F'$ of $K$.  Then 
$F,F'$ fall into the base case $t=0$ (since both contain $f_0$),
and hence have a gallery $F \rightarrow F'$, 
while $F',G$ contain the vertices $f_1,f_t$ that have
a path of length $t-1$ between them, and hence there is
a gallery $F'\rightarrow G$ by the inner induction on $t$.
Thus one has a gallery $F \rightarrow F'\rightarrow G$.
\end{proof}

Finally, we will need to know that some of the
preceding notions interact well with the notion of simplicial join.
Recall that for two abstract simplicial complexes
$\Delta_i$ on disjoint vertex sets $V_i$ for $i=1,2$, their
{\it simplicial join} on vertex set $V_1 \sqcup V_2$ has simplices
indexed by the ordered pairs $(F_1,F_2)$ 
in which $F_i$ is a simplex of $\Delta_i$ for $i=1,2$.
Here $\dim(F_1,F_2) = \dim F_1 + \dim F_2 + 1$.

\begin{proposition}
\label{prop:joins}
For two simplicial complexes $\Delta_1, \Delta_2$,
and their simplicial join $\Delta:=\Delta_1 * \Delta_2$,
\begin{enumerate}
\item[(i)] $\Delta_1, \Delta_2$ are both SGC if and only if
$\Delta$ is SGC, and
\item[(ii)] If $\Delta_i^{\langle j \rangle}$ are
connected for $1 \leq j \leq \dim \Delta_i$ and
$i=1,2$, then $\Delta^{\langle j \rangle}$ is
connected for $1 \leq j \leq \dim \Delta$.
\end{enumerate}
\end{proposition}
\begin{proof}

\noindent
{\sf (i):} 
First assume $\Delta_1 * \Delta_2$ is
SGC, and we show that $\Delta_1$ is SGC.
Given two faces $F,G$ of $\Delta_1$ (say with $\dim F \leq \dim G$)
choose any facet $H$ of $\Delta_2$ achieving
$\dim H = \dim \Delta_2$, and then since $\Delta_1 * \Delta_2$
was assumed to be SGC, there exists a sequential gallery
$(F,H) \rightarrow (G,H)$ in $\Delta_1 * \Delta_2$.  Because
of the hypothesis on $H$, one can check that every face $(F_i,H_i)$
appearing in this gallery has $\dim H_i = \dim H = \dim \Delta_2$.
This means that one can replace each $(F_i,H_i)$ with $(F_i,H)$
and (after eliminating duplicates) one still has a sequential
gallery $(F,H) \rightarrow (G,H)$, having $H$ constant in the second factor.
Projecting this sequential gallery onto its first factors $F_i$
then gives a sequential gallery $F \rightarrow G$.

Next, assume both $\Delta_i$ are SGC, and we show that
$\Delta_1 * \Delta_2$ is SGC.  Given two faces $(F_1,F_2), (G_1,G_2)$ of $\Delta$,
of the same dimension, one must exhibit a gallery connecting them.
Without loss of generality by relabelling, one may assume
$\dim F_1 \leq \dim G_1$.  Hence there exists a sequential
gallery $F_1 \rightarrow G_1$ which leads to a
sequential gallery 
\begin{equation}
\label{starting-gallery}
(F_1,F_2) \rightarrow (G_1,F_2).
\end{equation}

If also $\dim F_2 \leq \dim G_2$ (so that $\dim F_i = \dim G_i$ for
$i=1,2$) then there exists a gallery $F_2 \rightarrow G_2$, which 
leads to the gallery $(F_1,F_2) \rightarrow (G_1,F_2) \rightarrow (G_1,G_2)$.

If $\dim F_2 > \dim G_2$, choose a subface $F_2' \subseteq F_2$
with the same dimension as $G_2$.  Then there exists a
gallery $F_2' \rightarrow G_2$ leading to a gallery
$(G_1,F_2') \rightarrow (G_1,G_2)$.   Also, Lemma~\ref{lemma:sequential-gallery}
applied to \eqref{starting-gallery} 
implies the existence of a gallery $(F_1,F_2) \rightarrow (G_1,F_2')$.
Hence one has the gallery
$(F_1,F_2) \rightarrow (G_1,F_2') \rightarrow (G_1,G_2)$.

\vskip .1in
\noindent
{\sf (ii):} When either of the $\Delta_i$ is $(-1)$-dimensional,
that is, the complex consisting of only the empty face, then the
assertion is trivial.  So assume without loss of generality that
both are at least $0$-dimensional.  In particular, this makes
the join $\Delta$ always at least $1$-dimensional and connected.

Fixing $j$ in the range 
$$
1 \leq j \leq \dim \Delta = \dim \Delta_1 + \dim \Delta_2 +1,
$$
one can check that
$$
\Delta^{\langle j \rangle}
= \bigcup_{\substack{(j_1,j_2):j_1 + j_2 + 1 = j
\\ -1 \leq j_1 \leq \dim \Delta_1\\
\\ -1 \leq j_2 \leq \dim \Delta_2}}
\Delta_1^{\langle j_1 \rangle} * \Delta_2^{\langle j_2 \rangle}.
$$
Our hypotheses then imply that every term in the above union is
at least $1$-dimensional and connected.
Furthermore, if for $i=1,2$ one chooses
facets $F_i$  of $\Delta_i$ and vertices $v_i$ lying in $F_i$,
then the two vertices $(v_1,\varnothing), (\varnothing,v_2)$ of $\Delta$ have
$$
\begin{aligned}
(v_1,\varnothing) &\text{ lying in the terms with }j_1 \geq 0, \\
(\varnothing,v_2) &\text{ lying in the terms with }j_2 \geq 0.
\end{aligned}
$$
Furthermore, the edge $(v_1,v_2)$ of $\Delta$ lies inside
the facet $(F_1,F_2)$, and hence is an edge inside
$\Delta^{\langle j \rangle}$ connecting these two vertices.
This shows $\Delta^{\langle j \rangle}$ is connected.
\end{proof}

\subsection{Quadraticity}

\begin{definition}
\label{defn:quadraticity}
\rm \ \\
Given $R=\oplus_{d \geq 0} R_d$ an $\NN$-graded ring,
the tensor algebra 
$$
T_{R_0}(R_1) := \bigoplus_{d \geq 0} 
\underbrace{R_1 \otimes_{R_0} 
\cdots \otimes_{R_0} R_1}
_{T^d_{R_0}(R_1):=R_1^{\otimes_{R_0} d}} 
$$
is an $\NN$-graded ring.
One says that $R$ is {\it quadratic} if
the multiplication map 
$$
\begin{array}{rcl}
T_{R_0}(R_1)& \overset{\pi}{\longrightarrow} & R \\
r_1 \otimes \cdots \otimes r_d & \longmapsto & r_1 \cdots r_d.
\end{array}
$$
has the following two properties:
\begin{enumerate}
\item[(a)] It is surjective, that is, $R$ is generated as an $R_0$-algebra by $R_1$.
\item[(b)] Its kernel $\ker \pi$ is generated as a two-sided-ideal of $T_{R_0}(R_1)$
by its homogeneous component of degree two 
$$
\ker \pi \cap (R_1 \otimes_{R_0} R_1).
$$
\end{enumerate}
\end{definition}

We wish to determine which $\NN$-graded rings $R=\gr_I A$, where  
$(P,\sim)$ obey the axioms and $A=k[P]_\red$,
enjoy the Properties (a), (b).  
Abbreviating $T:=T_{A_0}(R_1)$,
note that the $A_0-A_0$-bimodule structure is respected
by $\pi$, and hence we can take advantage of both the
bimodule direct sum decomposition in Proposition~\ref{prop:tensor-bases}
and the $\NN$-grading to write
\begin{equation}
\label{tensor-decomposition}
T= 
\bigoplus_{d \geq 0} 
\qquad 
\bigoplus_{(\tilde{x},\tilde{y}) \in P/ \! \sim \times P/\! \sim}
\qquad 
\bigoplus_{\substack{\alpha \in \Int(P)/\! \sim:\\
                s(\alpha) = \tilde{x}\\ t(\alpha) = \tilde{y}}} 
\quad
T_{\alpha,d},
\end{equation}
where $T_{\alpha,d}$ has as $k$-basis the elements
$$
\bar{\xi}(m):=
\bar{\xi}_{\widetilde{[x_0,x_1]}} \otimes 
\bar{\xi}_{\widetilde{[x_1,x_2]}} \otimes 
\cdots \otimes 
\bar{\xi}_{\widetilde{[x_{d-1},x_d]}} 
$$
in which 
$$
m:=(x(\alpha)=x_0 < x_1 < \cdots < x_d =y(\alpha))
$$
is a {\it maximal} chain of length $d$ in the interval
$[x(\alpha),y(\alpha)]$. 

This allows us to quickly dispose of Property (a).

\begin{proposition}
\label{prop:generated-in-degree-one}
Let $P$ be a poset and $\sim $ an equivalence relation on $\Int(P)$, satisfying Axioms
\ref{ax:order-compatibility}, \ref{ax:invariance}, \ref{ax:finiteness}, \ref{ax:concatenation}.  
Then the reduced incidence algebra $A:=k[P]_\red$ always has $R=\gr_I A$
satisfying Property (a), that is, $R$ is always generated as an $R_0$-algebra
by $R_1$.
\end{proposition}
\begin{proof}
Given $\alpha$ with $d:=\ell(\alpha)$, pick
a maximal chain $m$ in $[x_0,y_0]$ achieving the maximum length $d=\ell(\alpha)$.
Then the element $\bar{\xi}_\alpha$ of $R_d$
can be expressed as the
product of the elements 
$\bar{\xi}_{\widetilde{[x_0,x_1]}} \cdot 
\bar{\xi}_{\widetilde{[x_1,x_2]}} \cdots
\bar{\xi}_{\widetilde{[x_{d-1},x_d]}}$,
each lying in $R_1$.
\end{proof}

\begin{theorem}
\label{thm:quadraticity}
Let $P$ be a poset and $\sim $ an equivalence relation on $\Int(P)$, satisfying Axioms
\ref{ax:order-compatibility}, \ref{ax:invariance}, \ref{ax:finiteness}, \ref{ax:concatenation}.  
Then the reduced incidence algebra $A:=k[P]_\red$ has
$R=\gr_I A$ quadratic
if and only if every open interval $(x,y)$ in $P$ 
has $\Delta(x,y)$ SGC.
\end{theorem}

By Proposition~\ref{prop:generated-in-degree-one},
one only needs to show that Property (b) in Definition \ref{defn:quadraticity} 
is equivalent to $P$ having every interval $[x_0,y_0]$
SGC.  

\begin{proof}
One has the following equivalences for $R=\gr_I A$ with $A=k[P]_\red$,
with justifications given below:
$$
\begin{aligned}
&R \text{ quadratic } \\
&\overset{(1)}{\Longleftrightarrow} 
\Tor_2^R(R_0,R_0)_j=0\text{ for all }j \neq 2 \\
&\overset{(2)}{\Longleftrightarrow} 
\text{For every interval }[x,y]\text{ in }P,\\
& \qquad
\text{ the complex }\Delta:=\Delta(x,y)\text{ has }
\tilde{H}_0(\Delta^{\langle j \rangle},\Delta^{\langle j+1 \rangle};k)=0
\text{ for } 1\leq j \leq \dim\Delta \\
&\overset{(3)}{\Longleftrightarrow} 
\text{For every interval }[x,y]\text{ in }P,\\
& \qquad
\text{ the complex }\Delta:=\Delta(x,y)\\
&\qquad
\text{ has } \Delta^{\langle j \rangle}\text{ connected for } 1\leq j \leq \dim\Delta \\
&\overset{(4)}{\Longleftrightarrow} 
\text{For every interval }[x,y]\text{ in }P,\\
& \qquad
\text{ the link }\Gamma\text{ of every face in the complex }\Delta:=\Delta(x,y)\\
& \qquad 
\text{ has }\Gamma^{\langle j \rangle}\text{ connected for } 1\leq j \leq \dim\Gamma \\
&\overset{(5)}{\Longleftrightarrow} 
\text{For every interval }[x,y]\text{ in }P,\\
& \qquad
\text{ the link }\Gamma\text{ of every face in the complex }\Delta:=\Delta(x,y)
\text{ is SGC}\\
&\overset{(6)}{\Longleftrightarrow} 
\text{For every interval }[x,y]\text{ in }P,\\
& \qquad
\text{ the complex }\Delta:=\Delta(x,y)\text{ is SGC}
\end{aligned}
$$

Equivalence (1), in the presence of Property (a), 
is well-known when $R_0$ is a field, and no harder when $R_0$ is semisimple; 
see \cite[Theorem 2.3.2]{BGS}.

Equivalence (2) follows from Theorem~\ref{thm:tor-computation}.

Equivalence (3) follows from Proposition~\ref{prop:connectivity-equivalences}.

Equivalence (4) follows from Proposition~\ref{prop:joins}(ii) and
the fact that the link of a face indexed by a chain
$$
x=x_0 < x_1 < \cdots < x_i = y
$$
in the complex $\Delta(x,y)$ is the simplicial 
join of the complexes $\Delta(x_{\ell-1},x_\ell)$.

Equivalence (5) again follows from Proposition~\ref{prop:connectivity-equivalences}.

Equivalence (6) follows from Proposition~\ref{prop:joins}(i) and
the same fact about links in $\Delta(x,y)$ that was used for Equivalence (4).
\end{proof}

\section{Subgroup lattices}
\label{sec:subgroups}

We discuss here for comparison another situation
where one has an algebraic condition
characterized in a pure/graded situation by
the Cohen-Macaulayness of some poset,
and more generally in the nonpure/nongraded situation
by sequential Cohen-Macaulayness.
It would be very interesting if these results had a common
generalization with Theorem~\ref{thm:main}.

We first recall two notions from group theory.

\begin{definition}
\rm \ \\
A {\it subnormal series} for a group $G$ is a tower of subgroups
$$
\{1\} = H_0 \triangleleft H_1 \triangleleft \cdots \triangleleft H_t = G.
$$
A {\it solvable} group $G$ is one having
a subnormal series in which each quotient $H_i/H_{i-1}$ is abelian.
A {\it supersolvable} group $G$ is one having
a subnormal series in which each quotient $H_i/H_{i-1}$ is cyclic,
and in addition, each $H_i$ is normal in $G$.
\end{definition}

\begin{theorem}(Bj\"{o}rner \cite[Theorem 3.3]{Bj}) 
A finite group $G$ is supersolvable if and only if its
lattice $L(G)$ of subgroups is Cohen-Macaulay over $k$
for some field $k$.
\end{theorem}

The following generalization is, essentially, the striking
result of J. Shareshian \cite[Theorem 1.4]{Sha}),
characterizing solvability via the {\it nonpure shellability} \cite{BW} of $L(G)$.
The authors thank Shareshian for discussions
on how to prove the following trivial extension of his result, and
for allowing them to include the proof sketch given here.

\begin{theorem}
A finite group $G$ is solvable if and only if its
lattice $L(G)$ of subgroups is sequentially Cohen-Macaulay over $k$
for some field $k$.
\end{theorem}
\begin{proof}
(Sketch)
For both implications, most of the work occurs
already in the proof of \cite[Theorem 1.4]{Sha}.

\vskip .1in
\noindent
($\Rightarrow$:)
For a finite solvable group $G$, the proof of \cite[Theorem 1.4]{Sha}
showed that $L(G)$ is nonpure shellable, a combinatorial
condition introduced by Bj\"orner and Wachs \cite{BW}, known
to imply sequential Cohen-Macaulayness over any field $k$.  

\vskip .1in
\noindent
($\Leftarrow$:)
The proof of \cite[Theorem 1.4]{Sha} showed that for $G$ not solvable, 
the lattice $L(G)$ cannot be nonpure shellable, using the following
plan, that we adapt here to show it cannot be sequentially Cohen-Macaulay 
over any field $k$.

One first notes that for subgroups $N \triangleleft H < G$,
the subgroup lattice $L(H/N)$ is isomorphic to the interval $[N,H]$ in $L(G)$.
As sequential Cohen-Macaulayness is a property inherited by intervals in a poset
(see Proposition~\ref{prop:BWW-poset-criterion} above), if $L(G)$ were
sequentially Cohen-Macaulay over $k$, the same would be true for $L(H/N)$.
This reduces one to checking that $L(G)$ is not sequentially Cohen-Macaulay when $G$
is a {\it minimal simple group}, that is, one which is nonabelian
and simple, but all of whose proper subgroups are solvable.

The proof of \cite[Theorem 1.4]{Sha}
used the classification of minimal simple groups $G$,
and checked case-by-case that none can have $L(G)$ nonpure shellable,
usually by an argument of a topological nature that
also precludes sequential Cohen-Macaulayness
over any field $k$.  

The only case that presents a problem is the minimal simple group 
$G=SL_3(\FF_3)$, where the argument given in
the proof \cite[Theorem 1.4]{Sha} is of a different nature,
more specific to contradicting the existence of a nonpure shelling.  Instead, for
this case $G=SL_3(\FF_3)$, we used the software package {\tt GAP} \cite{GAP} 
to check that the following numerology of $L(G)$ prevents it from being
sequentially Cohen-Macaulay over any field $k$.

Let $\Delta$ denote the order complex of the open interval 
$(\{1\},G)$ in $L(G)$ for $G=SL_3(\FF_3)$;  it is easily seen 
from the definitions that $L(G)$ is sequentially Cohen-Macaulay over $k$ if and only
if the same holds for $\Delta$.  Then
$\Delta$ has $6372$ vertices and is $6$-dimensional. 
If $\Delta$ were sequentially Cohen-Macaulay over $k$,
then by a theorem of Duval \cite[Theorem 3.3]{Duv}, 
its pure $6$-dimensional skeleton $\Delta^{[6]}$ (= the
subcomplex generated by all faces of dimension $6$)
would be Cohen-Macaulay over $k$. 
However, computation in {\tt GAP} shows that 
$\Delta^{[6]}$ has $f$-vector and $h$-vector (see \cite[Chapter 5]{BH}, 
\cite[\S II.1,2]{St})
$$
\begin{array}{ccccccccc}
f&=(f_{-1},&f_0,&f_1,&f_2,&f_3,&f_4,&f_5,& f_6) \\
 &=(1, &2561,& 49686, &228904, &456846, &472368, &252044, &54600) \\
 &     &     &        &        &        &        &        & \\
h&=(h_0,&h_1,&h_2,&h_3,&h_4,&h_5,&h_6,&h_7)\\
 &=(1, &2554,& 34341, &18854,  &-13095, &16788,  &-4699,  &-144).
\end{array}
$$ 
Negative entries in the $h$-vector imply that $\Delta^{[6]}$ 
is not Cohen-Macaulay (see \cite[Theorem 5.1.10]{BH}, \cite[Theorem II.3.3]{St}) 
over any field $k$.  Alternatively, computation in {\tt GAP} shows that 
$
\widetilde{H}_i(\Delta^{[6]};\mathbb{Z})=\ZZ^{144}
$
for $i=3$ and vanishes for all other $i$.
Reisner's criterion (\cite[Corollary 5.3.9]{BH}, \cite[Corollary II.4.2]{St}) 
then implies that $\Delta^{[6]}$ is not Cohen-Macaulay over any field $k$.
\end{proof}

\end{document}